\documentclass{amsart}
\usepackage{amsfonts,amscd, amssymb}
\usepackage{tikz-cd}
\newtheorem{theorem}{Theorem}[section]
\newtheorem{lemma}[theorem]{Lemma}
\newtheorem{corollary}[theorem]{Corollary}
\newtheorem{proposition}[theorem]{Proposition}
\theoremstyle{remark}

\theoremstyle{definition}

\numberwithin{equation}{section}
\makeatother
\newcommand{\I}{1\!{\mathrm l}}
\newcommand{\norm}[1]{\Vert#1\Vert}

\DeclareMathOperator{\Bdb}{{\mathbb B}}
\DeclareMathOperator{\bB}{{\mathbb B}}
\DeclareMathOperator{\bK}{{\mathbb K}}
\DeclareMathOperator{\bR}{{\mathbb R}}

 \DeclareMathOperator{\Kdb}{{\mathbb K}}
\DeclareMathOperator{\Cdb}{{\mathbb C}}
\DeclareMathOperator{\Rdb}{{\mathbb R}}

\DeclareMathOperator{\Hdb}{{\mathbb H}}

\DeclareMathOperator{\Fdb}{{\mathbb F}} 

\DeclareMathOperator{\Ndb}{{\mathbb N}}

\DeclareMathOperator{\A}{{\mathcal A}}

\DeclareMathOperator{\cM}{{\mathcal M}}

\DeclareMathOperator{\cL}{{\mathcal L}}
\DeclareMathOperator{\cA}{{\mathcal A}}
\DeclareMathOperator{\cT}{{\mathcal T}}
\DeclareMathOperator{\cF}{{\mathcal F}}
\DeclareMathOperator{\cS}{{\mathcal S}}
\DeclareMathOperator{\cU}{{\mathcal U}}

\begin{document}


\title[Real operator spaces and algebras]{Real operator spaces and operator algebras}
\author[D. P. Blecher]{David P. Blecher}
\address{Department of Mathematics\\ University of Houston\\ Houston, TX
77204-3008, USA}
\email{dpbleche@central.uh.edu}

\date{}

\begin{abstract}  We verify that a large portion of the theory of complex 
operator spaces and operator algebras (as represented by the 2004 book  {\rm \cite{BLM}} for specificity)
transfers to the real case.  We point out some of the results that do not work in the real case. 
We also discuss how the theory and standard constructions interact with the 
complexification, which  is often as important but sometimes much less obvious.
 For example, we develop the real case of the theory of operator space multipliers and the operator space  centralizer 
 algebra,  and discuss how these topics connect with the complexification. 
 This turns out to differ in some important details from the complex case. 
   We also characterize real structure in complex operator spaces; and give `real' characterizations of some of the most important 
   objects in the subject. 
  \end{abstract}
 
 
\maketitle

\section{Introduction}  
Complex operator spaces are an important 
umbrella category containing $C^*$-algebras, operator systems, operator algebras, von Neumann algebras,
and many other objects of interest in modern analysis and also in modern quantum physics (such as quantum information theory).   They have an extensive theory
(see e.g.\ the texts \cite{ER,Pisbk,Pau,BLM}), and have very important applications in all of these subjects. 

Ruan initiated the study of real operator spaces in \cite{ROnr,RComp}, and this study was continued in 
\cite{Sharma, BT}.   
A real operator space may either be viewed as a real subspace of $B(H)$ for a real Hilbert space $H$, or abstractly as 
a vector space with a norm $\| \cdot \|_n$ on $M_n(X)$ for each $n \in \Ndb$ satisfying  the conditions of
Ruan's characterization in  \cite{ROnr}.    Real structure occurs naturally and crucially
 in very many areas of mathematics, as is mentioned also for example in the first
paragraphs of \cite{BT} and \cite{Sharma}, or in \cite{Ros}; and also shows up in places in modern mathematical 
quantum physics (see e.g.\ \cite{Ch2} and references
therein).  Unfortunately there is not very much at all on real operator spaces in the literature;  the works cited at the start of the paragraph  do not add
up to a lot of pages.  
Even  for Hilbert spaces the very rough comparisons with the amount of  literature in the complex case  
 mentioned in the introductions of \cite{BPi,BT} are startling.  A researcher working on a problem which involves real operator spaces or systems,  would have had to 
reconstruct a large amount of the theory from scratch in the real case. 
Thus for contemporary applications in the areas of mathematics and physics mentioned above, it is of interest 
to understand the real case of the important results in (complex) operator space theory--what works and how it is
connected to the complex case.   

 This is 
one goal of the present paper, to supply in some fashion such a resource.   Since this is a daunting and  not well defined task we restrict 
ourselves in the last part of our paper mostly to the more modest target 
of checking  the real case of the most important chapters in the text \cite{BLM}, and how the 
facts and structures there relate to the complexification, which is sometimes quite nontrivial. 
   To not try the readers patience we have attempted to be brief. 
Our proofs are often deceptively short, often referencing deep results and arguments.   Some complementary results
can be found in the companion paper \cite{BCK}.

We will begin our paper with several fundamental applications.  
 Section \ref{absc} establishes   the real case of  some of the most important 
characterizations of objects of particular interest in operator spaces: operator algebras, operator modules, unital 
operator spaces, and operator systems. 
In Section \ref{csiros}  we characterize when a real operator space may be given a 
complex structure.  For example, as we observe, the quaternions are a real operator space  and a complex Banach space, 
but are not a complex operator space.  We explain why this happens and what is needed to remedy it.   
We show
also in this section that in contrast to the Banach space case (see e.g.\ \cite{FG}),
every complex operator space $X$  has a unique complex operator space  structure up to complete isometry, and moreover such structure has a very simple classification. 

Section \ref{ospm}  is devoted to 
extending 
to  real operator spaces the deeper aspects of the theory of operator space multipliers and  operator space  centralizers (see \cite[Section 7]{BZ} for 
the latter in the complex case).    
This is one of the more profound parts of the `completely isometric theory' of operator spaces.  
  Some of this is needed in \cite{BK} which uses such multipliers and  operator space  centralizers in the real case.
Indeed  it will be applicable,  
and probably critical,  in situations in the future involving the real case
of operator modules in the sense of Christensen and Sinclair, or more generally involving 
 subspaces of $B(H)$ that are invariant under left multiplication by various operators on $H$.
The real theory differs in a few  
important details from the complex case.   Some things are considerably more difficult in the real case.  For example a hard question in the real case is the characterization of $M$-ideals in real TRO's; see a paper referred to in the Acknowledgements below. 
 It does not apply to some operator spaces (some spaces have no interesting 
$M$-ideals or  multipliers), but is a very powerful tool in spaces that do possess some `operator algebraic structure' in a loose sense. For example Section \ref{csiros} shows how multipliers are 
a main ingredient in the complex structure of real operator spaces.  The centralizer is a generalization of the center of a $C^*$-algebra, and is also often key to understanding
the ideal structure (or $M$-ideal structure in general settings).  Thus one may expect the centralizer to 
 play a role in future generalizations of ideal structure and centers, and possibly in directions 
such as Bryder's result relating  the intersection property to the action on the centre of the injective envelope \cite{Bry}.

Although there are other things there, Sections \ref{os}--\ref{last} 
mainly verify in a very economical format
the real case of the remaining theory in Chapters 1--4 and 8 of \cite{BLM}, and establish how the basic constructions there 
interact  with the complexification. We also check some selected results in Chapter 5 there.  
The Appendix of \cite{BLM}  consists of standard facts in 
functional analysis, the real case of almost all of which are well known.  The few that are unclear in the real case we have discussed in scattered locations below.  
Since the chapters in \cite{BLM} build on each other, we will systematically start from the beginning
(avoiding of course results already in the literature). 
As one would expect, many results in the real theory are proved just as in the complex case, 
We say almost nothing about such results; although we may sometimes mention a potentially confusing step in the proof.
Similarly, many 
  of the results follow swiftly by complexification, but sometimes we give additional details on how this should be done. 
Then some results do require new proofs, usually because the complex argument 
involves facts that fail in the real case.  Indeed in addition to the `real issues' listed in the Introductions to \cite{BT,BCK}, 
 arguments that involve selfadjoint or positive elements and their span, or the polarization identity, often fail in the real case. 
 We also point out the results that cannot be made to work in the real case.

Thus for example Section \ref{last} may be viewed in some sense as a 
(very economical) complete theory of real $C^*$-modules, and in particular the real operator space aspects of that subject. 
 Sections \ref{os}--\ref{om} verify in the real case aspects of the general theory of operator spaces (and in particular their duality
 and tensor products), 
 operator algebras, and operator modules.   These sections also establish 
 functoriality of the complexification (see the next paragraph) for many important constructions (such as their tensor products), and develop some other  new aspects of the complexification.
 Section \ref{inj4} mostly concerns  the real case of 
 the few remaining 
 topics from Chapter 4 of \cite{BLM}.

As we have alluded to earlier, not only do we want to check that 
the real versions of the complex theory  work, 
 we also want to know what the complexifications are 
of standard constructions, and this is often as important but much less obvious.   
For example, it is important to know that the complexification of a particular operator space tensor product is 
a particular  tensor product of the complexifications.
More generally, it is important to know for which `constructions' $F$ in the theory we have $F(X)_c = F(X_c)$ canonically completely isometrically.  In some cases 
one has to be careful with the identifications.  For example, just because one has proved that a complex space $W$ is complex linearly completely isometrically isomorphic to $X_c$, 
 and that $W$ is a reasonable complexification of a real space $Y$, one cannot conclude 
 that then $Y \cong X$. In fact $Y$ may not be isometric to $X$. 
 This fails even if $X, Y, W$ are finite dimensional $C^*$-algebras (even if $X = M_2(\Rdb), W = M_2(\Cdb)$).  
  Or as another example,  if we have  a complex space (such as a von Neumann algebra) with a unique operator space predual, and that predual is a reasonable
 complexification of a real space, then that real space need not be unique up to complete isometry (or even isometry).

We now turn to notation.  The reader will need to be familiar with the  basics of complex operator spaces and von Neumann algebras  
as may be found in early chapters of \cite{BLM,ER,Pau, Pisbk}, and e.g.\  \cite{P}. 
It would be helpful to also browse the (small) existing real operator space theory  \cite{ROnr,RComp,Sharma,BT}.  Some basic 
real $C^*$-algebra theory may be found in \cite{Li} or \cite{ARU,Good}.
 The letters $H, K$ are reserved for real Hilbert spaces.  Every complex Hilbert space is a real Hilbert space with the `real part' of the inner product. 
 We sometimes write the complex number $i$ as $\iota$ to avoid confusion with matrix subscripting. 
 For us a {\em projection}  in an algebra 
is always an orthogonal projection (so $p = p^2 = p^*$).    A  normed algebra $A$  is {\em unital} if it has an identity $1$ of norm $1$, 
and a map $T$ 
is unital if $T(1) = 1$.  We say that $A$ is {\em approximately unital} if it has a contractive approximate identity (cai).
 We write $X_{\rm sa}$ for the selfadjoint operators 
 in $X$.   
 
 An operator space $X$ comes with a norm $\| \cdot \|_n$ on  $M_n(X)$.
 Sometimes the sequence of norms $(\| \cdot \|_n)$ 
 is called the {\em operator space structure}. 
If $T : X \to Y$ we write $T^{(n)}$ for the canonical `entrywise' amplification taking $M_n(X)$ to $M_n(Y)$.   
The completely bounded norm is $\| T \|_{\rm cb} = \sup_n \, \| T^{(n)} \|$, and $T$ is 
completely  contractive if  $\| T \|_{\rm cb}  \leq 1$. 
A map $T$ is said to be {\em positive} if it takes  positive elements to positive elements, and  {\em 
completely positive} if $T^{(n)}$ is  positive for all $n \in \Ndb$. A UCP map is  unital and completely positive.

 An {\em operator space complexification} of a real operator space $X$ 
is a pair $(X_c, \kappa)$ consisting of a complex operator space $X_c$ and a real linear complete isometry $\kappa : X \to X_c$ 
such that $X_c = \kappa(X) \oplus i \, \kappa(X)$ as a vector space.   For simplicity we usually identify $X$ and $\kappa(X)$ and write $X_c = X + i \, X$.
We say that the complexification is {\em reasonable} if the map $\theta_X(x+iy) = x - iy$ on $X_c$ (that is
$\kappa(x) + i \kappa(y) \mapsto  \kappa(x) - i \kappa(y)$ for $x, y \in X$), is 
a complete isometry.  Ruan proved that a  real operator space has a unique reasonable complexification
$X_c = X + i X$ up to complete isometry \cite{RComp}.       Sometimes we will call this a
`completely reasonable complexification', or a `reasonable operator space complexification'.  
Conversely, if $X$ is a real operator space with a complete isometry $\kappa : X \to Y$ into a complex operator space,   
 then $(Y,\kappa)$ is a reasonable operator space complexification  of $X$ if and only if $Y$ possesses a 
conjugate linear completely isometric period 2 automorphism whose fixed points are $\kappa(X)$ (\cite[Theorem 3.2]{RComp}). 
We will use the latter result repeatedly, as well as the notation $\theta_X$. 

We recall that the complexification may be identified up to real  complete isometry with the operator subspace $V_X$ of $M_2(X)$ 
consisting of matrices of the form 
\begin{equation} \label{ofr} \begin{bmatrix}
       x    & -y \\
       y   & x
    \end{bmatrix}
    \end{equation} 
    for $x, y \in X$. 
    
     We will need the fact that 
  $M_{n,m}(X)_c = M_{n,m}(X_c)$ completely isometrically for an operator space $X$.
  This  may be seen in several ways, for example by using the identification of $X_c$ with $V_X$ above.
  Or it may be proved by noting that it is sufficient to assume $m = n$ and $X = B(H)$.  And for a $C^*$-algebra $B$, we have  
  $M_n(B)_c \cong   M_n(B_c)$ $*$-isomorphically so isometrically.  
  If  $T : X \to Y$ is 	a real completely bounded map then $T_c(x+iy) = T(x) + i T(y)$ for $x, y \in X$.
Ruan shows in \cite[Theorem 2.1]{RComp} that $\| T_c \|_{cb} = \| T \|_{cb}$.  We include a 
short proof of this.  
 
  \begin{proposition} \label{cbcmpx}
	 If $T : X \to Y$ is 
	a real completely bounded (resp.\ completely  isometric, complete quotient) linear operator between real operator spaces, then $T_c$ is completely bounded (resp.\
	 a complete isometry, complete quotient), with $\norm{T_c}_{cb}=\norm{T}_{cb}$.
\end{proposition}

\begin{proof}  Clearly $\norm{T}_{cb}\leq\norm{T_c}_{cb}$ since $\norm{T}_{cb}=\norm{T_c\arrowvert_X}_{cb}$. For the other inequality,  by the discussion around
(\ref{ofr}) we can identify $\| T_c^{(n)}([x_{ij}+\iota \, y_{ij}]) \|$ with the norm of the matrix 
$$\begin{bmatrix}
			T(x_{ij}) & -T(y_{ij})\\
			T(y_{ij}) & T(x_{ij})\\
		\end{bmatrix} , $$
 for $[x_{ij}+\iota \, y_{ij}]\in M_n(X_c)$.  This quantity is 
 dominated by
  $\| T \|_{cb}$ times the norm of the matrix  in (\ref{ofr}), which is 
 $\| [x_{ij}+\iota \, y_{ij}] \|$.
 Hence $T_c$ is completely bounded and
  $\norm{T_c}_{cb}\leq\norm{T}_{cb}$. So $\norm{T_c}_{cb}=\norm{T}_{cb}$. If $T$ is a complete isometry, then the matrix in the displayed equation 
   has the same norm as  the matrix  in (\ref{ofr}), so  that $T_c$ is a complete isometry. 
  The `complete quotient' assertion is generalized in \cite[Proposition 5.5]{BK}).
  \end{proof}

We showed in \cite[Section 2]{BCK} that $CB(X_c,Y_c)$ is a  reasonable complexification of $CB(X,Y)$.  
In places the reader will also need to be familiar with the theory of the injective envelope $I(X)$.  In the complex case this may be found 
in e.g.\ \cite{BLM,ER,Pau}.   The real case was initiated in \cite{Sharma}, and continued in \cite{BCK} (see also Section \ref{inj4}
although this is not used much in sections before that). 

A real  {\em unital operator space} is a  operator space  $X$ with a distinguished element $u \in X$ and a 
 real complete isometry $T : X \to B(H)$ with $T(u) = I_H$.  
 A real  {\em operator system} is a  unital operator space  $X$ with an involution  and a 
 real complete isometry $T : X \to B(H)$ with $T(u) = I_H$ which is selfadjoint (that is $T(x^*) = T(x)^*$). 
The diagonal $\Delta(X) = X \cap X^*$ of a unital real operator space $X$ is a well-defined real operator system  independent 
of representation  as in the complex case.  This follows e.g.\ by  
the discussion after Corollary 2.5 in \cite{BT}: suppose that
 $T : X \to Y$ is a surjective unital complete isometry between real unital operator spaces with
$I_H \in X \subset B(H)$ and $I_K \in Y \subset B(K)$.  Then the canonical extension $\tilde{T}: X + X^*  \to B(K) : x + y^*  \mapsto T (x) + T (y)^*$ is well defined
for $x,y \in X$, is 
selfadjoint and is a completely isometric complete order embedding onto $Y + Y^*$.  
(Whether this is isometric when $T$ is a unital  isometry was posed  in \cite{BT} when $X$ and $Y$ are in addition operator algebras.
In fact this is false even in the complex case.  Simple counterexamples may be manufactured 
using the  $\cU(X)$ construction from Section 2.2  of \cite{BLM}.) We see that $T(X \cap X^*) = Y \cap Y^*$, and that $T$ is
selfadjoint  on $X \cap X^*$.  
Indeed the results in 1.3.4--1.3.7  in 
\cite{BLM} 
hold in the real case, as does Choi's result Proposition 1.3.11 and 1.3.12 and 
the Choi-Effros Theorem 1.3.13 there (by going to the complexification if necessary).  Most of this was established in \cite{Sharma,BT} and 
  \cite[Section 4]{ROnr}.    Example 5.10 in \cite{BT}  shows that 1.3.8  in 
\cite{BLM} 
fails in the real case, although it is true for unital complete contractions (by e.g.\ 2.4 and  2.5 in \cite{BT}).  For the Paulsen system we have
$\cS(X)_c \cong \cS(X_c)$; indeed this is inherited from the relation $M_2(B(H))_c \cong M_2(B(H)_c)$.   

Perhaps shockingly, the complexification $A_c$ of an (even unital) operator algebra $A$ need not be well defined up to isometric (as opposed to complete isometric) isomorphism.  For that reason operator algebras and their complexification are almost always treated here in the operator space setting.
The diagonal $\Delta(A) = A \cap A^*$ of a real operator algebra $A$ in $B(H)$
(with possibly no kind of identity) is a well-defined operator algebra independent 
of representation  as in the complex case.  Indeed by \cite[Theorem 2.6]{BT}, 
 a contractive homomorphism $\pi : A \to B$ between operator algebras takes the diagonal $\Delta(A)$ 
$*$-isomorphically onto a closed $C^*$-subalgebra of $\Delta(B)$.   Thus 
 the results in  2.1.2 in \cite{BLM} hold in the real case. For an operator algebra or unital operator space we have that 
$\Delta(X_c) = X_c \cap X_c^* = \Delta(X)_c$ in $B(H)_c$.
 Because the selfadjoint elements do not necessarily span a real 
$C^*$-algebra 
 (or may be all of the $C^*$-algebra) one needs to be careful in places.  
Thus $\Delta(A)$ need not be the span of the selfadjoint elements 
 (nor of the projections if $A$ is also weak* closed and not commutative), unlike in the complex case.

\section{Abstract characterizations}  \label{absc} Complex operator space theory began with Ruan's abstract characterization of complex operator spaces \cite[Theorem 2.3.5]{ER}.
Ruan also gave the matching characterization of real  operator spaces in \cite{ROnr}. 
Similarly there is a well known  abstract characterization of complex operator  algebras with contractive approximate identity (cai) \cite[Theorem 2.3.2]{BLM},
and its matching real version is in \cite{Sharma}.   
Real `dual operator algebras'  are characterized in Theorem \ref{doa} below.   
 Nonunital operator algebras are characterized up to completely bounded isomorphism in \cite[Section 5.2]{BLM}, and the real case follows immediately  by complexification.
 Nonunital real operator algebras  may be characterized up to complete isometry   as follows.

\begin{theorem}  \label{KP}  {\rm (Real version of the Kaneda–Paulsen theorem)}\  Let $X$ be a real operator space. The real algebra products on $X$ for which there exists a real linear completely isometric homomorphism from $X$ onto a real  operator algebra, are in a bijective correspondence with the elements $z \in  {\rm Ball}(I(X))$ such that 
$X z^* X \subset  X$ in the ternary product of $I(X)$ (recall that the injective envelope $I(X)$ is a real ternary subsystem of $I(X_c)$  {\rm \cite[Section 4]{BCK}}).
For such $z$ the associated operator algebra product on $X$ is $x z^*  y$. \end{theorem} 

\begin{proof}  As in the proof in \cite[Theorem 5.2]{BNmetric2} the one direction, and the last statement, follows from Remark 2 on p. 194 of \cite{BRS}, viewing $I(X)$ as a ternary system in
$B(H)$, and taking $V$ there to be $z^*$.  
For the other direction, suppose that  $X$ is a real  operator algebra, and that (by the complex version of the present theorem)
 the canonical operator algebra product on $X_c$ is given by 
$x z^* y$ for some $z  \in  {\rm Ball}(I_{\Cdb}(X_c)) = {\rm Ball}(I_{\Cdb}(X)_c)$ (see \cite[Section 4]{BCK} for the last identification). Write $z = z_1 + i z_2$ with $z_i \in I(X)$. 
Since $X z^* X \subset X$ it follows that  $X z_2^* X = 0$.   Therefore $X_c  z_2^* X_c = 0$ and $z_2 = 0$ by \cite[Theorem 4.4.12]{BLM}.
So $z  \in  {\rm Ball}(I(X))$.  The bijectivity follows similarly from \cite[Theorem 4.4.12]{BLM} as in the proof of \cite[Theorem 5.2]{BNmetric2}.
 \end{proof}

 {\bf Real unital operator spaces:}  Appropriate variants of the  characterizations of  unital operator spaces from \cite{BNmetric,  BNmetric2} hold 
 in the real case.   
 
 \begin{theorem}  \label{chu}   If $X$  is a real  operator space and $u \in X$ with $\Vert u \Vert = 1$.  The following  are equivalent: 
 \begin{itemize} 
\item [(1)] $(X, u)$ is a real unital operator space.
\item [(2)] $\| [ u_n \; \, x ] \| \geq  \sqrt{2}  , \; \; \; \;  \big\|  \left[ \begin{array}{cl}
u_n  \\
x \end{array} \right] \big\| \geq  \sqrt{2} , \qquad n \in \Ndb, x \in M_n(X), \| x \| = 1.$
\item [(3)] $\Big\|  \left[ \begin{array}{cl} u_n & - x \\
 x & u_n \end{array} \right]  \Big\|  \geq \sqrt{1 + \| x \|},$ for all $n \in \Ndb, x \in M_n(X)$. \end{itemize} 
 Here $u_n$ is the diagonal matrix $u \otimes  I_n$ in $M_n(X)$ with $u$ in each diagonal entry.
 \end{theorem} 

\begin{proof}   (1) $\Rightarrow$ (2) \ This is  the easy direction, and follows from the real $C^*$-identity (as in \cite{BNmetric}).

(1) $\Rightarrow$ (3) \ Assuming (1), $(X_c, u)$ is a complex unital operator space.  
The left side of (3) is unchanged if we replace $x$ by $-x$ (as may be seen be pre- and post multiplying by the diagonal matrix with entries 
$1, -1$).   
Thus by Lemma 2.2  in \cite{BNmetric}
the left side of (3) (with respect to $X_c$) 
is $\max \{ \| u_n + i^k x \| : k = 0,1, 2, 3 \}$.  This is $\geq \sqrt{1 + \| x \|}$ by  \cite[Theorem 1.1] {BNmetric}.

 (2) $\Rightarrow$ (1) \ Let $v : X_c \to M_2(X)$ be the complete isometry yielding the matrix in (\ref{ofr}).
  If $x \in M_n(X_c) = M_n(X)_c, \| x \| = 1,$ then  $\| v^{(n)}(x) \| = 1$.  Applying $v$ we get 
$$\| [ u_n \; \, x ] \| = \| [ u_{2n} \; \, v^{(n)}(x) ] \| \geq  \sqrt{2} .$$ 
Thus by the characterization in \cite{BNmetric} there exists a linear complete isometry $T :  X_c \to B(H)$ with $T(u) = I_H$. 
Restricting to $X$ we get the result.   (We are also using that complex Hilbert spaces are real Hilbert spaces, as we do in many places in this paper).

 (3) $\Rightarrow$ (1) \ This is somewhat similar.   Let  $x \in M_n(X_c)$ with $x = y + iz$ for $y, z \in M_n(X)$.
 By reversing the argument for (1) $\Rightarrow$ (3) above we have that $$\max \{ \| u_n + i^k x \| : k = 0,1, 2, 3 \} = \Big\|  \left[ \begin{array}{cl} u_n & - x \\
 x & u_n \end{array} \right]  \Big\|.$$ Applying $v$ as in the last paragraph, this equals
 $$\Big\|  \left[ \begin{array}{cl} u_{2n} & - v^{(n)}(x) \\
  v^{(n)}(x) & u_{2n} \end{array} \right]  \Big\| \geq \sqrt{1 + \| v^{(n)}(x) \|} = \sqrt{1 + \| x \|}.$$
 Thus $(X_c, u)$ is a complex unital operator space by  \cite[Theorem 1.1]{BNmetric}, and we finish as in the last paragraph.   \end{proof} 
 
Most of the other characterizations in the papers \cite{BNmetric, BNmetric2}  hold in the real case, for example the characterization of isometries, coisometries and unitaries.  The characterization of these objects in Theorem 2.4 of [8] in the real case follows by a modification of an argument in the last proof. Also,
  Lemma 5.1, Theorem 5.5 and Corollary 5.6 in \cite{BNmetric2}  hold in the real case.  
  We plan to give more details elsewhere (See work referred to in the Acknowledgements below). 
    Thus unital real operator algebras can be characterized as real operator spaces $X$ with a 
  coisometry $u$ and a bilinear map  $m : X \times X \to X$ 
 such that  $m(x,u) = x$ for  $x \in X$, and 
 such that
 $$\left\| \left[ \begin{array}{cl} m(x,a_{ij}) \\ 
b_{ij} \end{array} \right] \right\| \leq \left\| \left[ \begin{array}{cl} a_{ij} \\ 
b_{ij} \end{array} \right] \right\| \, \quad 
[a_{ij}], [b_{ij}] \in M_n(X) , $$ 
for $n \in \Ndb$ and $x \in {\rm Ball}(X)$.  This `improves'  the characterization of unital operator algebras 
mentioned in the first paragraph of this section.  
 
 \medskip
 
{\bf Real operator systems:} 
The metric characterizations of operator systems in e.g.\  \cite[Theorem 3.4]{BNmetric} and \cite[Proposition 4.2]{BNmetric2}
 are valid in the real case by the same proofs.     We state only  the latter characterization 
 in the real case.  We note that this uses the real case of the proof of the former characterization (i.e.\ \cite[Theorem 3.4]{BNmetric}), which 
 characterization does not even involve an involution on the space.

  \begin{theorem}  \label{BNm2}   Let $X$ be a real  operator space possessing a real linear period 2 involution $*$ 
  with $\| x^* \|  =  \| x \|$ for $x = [x_{ij}] \in M_n(X), n \in \Ndb$. Here $x^* =  [x_{ji}^* ]$. 
   Let  $u \in {\rm Ball}(X)$ with $u = u^*$.  Then   there exists a real $*$-linear complete isometry from $X$ onto a real operator system  with $T(u) = 1$,  if and only if 
 $$\Big\|  \left[ \begin{array}{cl} u_n &  x \\
 -x^* & u_n \end{array} \right]  \Big\|  \; = \;  \sqrt{1 + \| x \|^2}, \qquad n \in \Ndb, x \in M_n(X).$$  
 Here $u_n = u \otimes  I_n$ as in the Theorem {\rm \ref{chu}.}
 \end{theorem} 
 
 \begin{proof}   We adapt the complex  proof  of \cite[Proposition 4.2]{BNmetric2}.   
 This proof is quite complicated, so that even with applying the hints below it may take the reader some time to master this adaption 
 in conjunction with a careful study of each step of  \cite[Proposition 4.2]{BNmetric2}.  
   As stated above, this adaption uses the real version  of  the proof of the characterization \cite[Theorem 3.4]{BNmetric}, but setting $y = y_n = -x^*$ there, where $*$ is the involution on $X$.
 This proof involves suprema over states on $M_2(B(H))$.  We may use vector states here, using the fact that 
 vector states norm  elements with $a = a^*$ even in the real case.  (The latter adjoint  is in $M_2(B(H))$, and is 
 not yet related to the involution in the statement of the 
 present theorem.)   We also need to appeal to  
  Theorem \ref{chu}  (2) above in place of its complex variant in 
 the proof 
 of \cite[Proposition 4.2]{BNmetric2}.     \end{proof}

 There are characterizations of complex operator systems in terms of the 
 positive cones in $M_n(X)$ in the literature, most notably the Choi-Effros characterization \cite{CE} (but see also e.g.\ \cite[Section 7]{KKM} for a recent characterization in terms
 of noncommutative convexity/quasistates).  We will not discuss these here except to say that in the real case 
 an operator system may have no nontrivial positive elements so that the obvious formulations of such a result in the real case fail.
 For this reason no doubt some authors study a subclass of the operator systems which do have large positive cones; however this would
 exclude some of the most interesting real unital $C^*$-algebras.
 
 \medskip

 {\bf Real operator modules:}   The real version of the Christensen-Effros-Sinclair theorem characterizing 
operator modules and bimodules (the complex version may be found e.g.\ in \cite[Theorem 3.3.1]{BLM}) holds easily
by complexification: 

\begin{theorem} \label{5CES}
Let $A$ and $B$ be approximately unital real operator algebras, and let
$X$ be a real  operator space which is a nondegenerate bimodule
over $A$ and $B$, satisfying 
$$\| \alpha x \beta \| \leq \| \alpha  \|  \|  x  \|  \|  \beta \| , \qquad n \in \Ndb, x \in M_n(X), \alpha \in M_n(A), \beta \in M_n(B).$$
  Then there exist real Hilbert spaces $H$ and $K$,  a real linear 
completely isometric  map $\Phi \colon  X \rightarrow 
B(K,H)$, and completely contractive nondegenerate real linear representations 
$\theta$ of $A$ on $H$, and $\pi$ of $B$ on $K$, such that 
$$\theta(a) \Phi(x) = \Phi(ax)
\; \; \; \; \; \text{and} \; \; \; \; \;
\Phi(x) \pi(b) = \Phi(xb),
\qquad a \in A, b \in B, x \in X.$$ Thus $X$ is completely $A$-$B$-isometric to the
concrete operator $A$-$B$-bimodule $\Phi(X)$. Moreover, $\Phi$,
$\theta, \pi$ may chosen to all be completely isometric, and such
that $H = K$.  If $A = B$ then one may choose,  in addition to all the above,
$\pi = \theta$.
\end{theorem}

\begin{proof}  It is easy to see that $X_c$ is an $A_c$-$B_c$-bimodule algebraically.
Claim:   $\| \alpha x  \| \leq \| \alpha  \|  \|  x  \|$ for $x \in M_n(X_c), \alpha \in M_n(A_c)$.
 This follows  by a matrix calculation whose details we suppress since it is similar to the first 5 lines of the proof that (ii) implies (iii) in \cite[Theorem 3.2]{Sharma}.
 Similarly, $\| x \beta \| \leq   \|  x  \|  \|  \beta \|$.  Hence  the displayed inequality as stated holds in the complexified spaces.
 Thus  $X_c$ is a nondegenerate bimodule
over $A_c$ and $B_c$ satisfying the conditions of the complex version of the theorem. 
Our result  follows from that version, viewing all complex spaces  as real spaces, 
and restricting the ensuing $\Phi, \theta, \pi$ to  real linear maps on $X, A, B$ respectively, as in earlier proofs.  
 \end{proof}

\section{Complex structure in operator spaces} \label{csiros}

This section could be read after Section \ref{ospm} in one sense, since the proofs (not the theorem statements) here 
use the perspective of that theory.  However we have placed it here  in view of its importance,
 and because of the sheer length and technicality
of Section \ref{ospm}.    We will summarize in a more or less selfcontained way before the first theorem
the few results we need from \cite{Sharma}  for the proofs, but some readers unfamiliar with the existing
complex theory of multipliers will possibly need to look at  \cite{Sharma}  or Section \ref{ospm}  at times during the proofs. 

By a {\em complex operator space  structure}
(resp.\  {\em complex Banach space  structure}) on a real operator (resp.\  Banach) space $X$ we mean an action of $i$ on $X$ as a real linear 
map 
 $J : X \to X$ such 
that $X$  (with unchanged norms) is a complex operator (resp.\  Banach)  space with scalar product $(s + it) x = s x + t J(x)$ for $x \in X$ and $s,t \in \Rdb$. 
Any complex operator space 
of course has canonical complex operator space  structure corresponding to $\theta(x) = ix$ for $x \in X$, in the notation above. 
It is well known and obvious that a map $J : X \to X$ on a real Banach space corresponds to a complex Banach space  structure on $X$ 
(with $ix = Jx$) if and only if $J^2 = -I$ and $x \mapsto s x + t J(x)$ is an isometry 
whenever  $s = \cos \theta, t = \sin \theta$, that is whenever $s^2 + t^2 = 1$.   Write $u_\theta = I \, cos \theta + J \, \sin \theta$.
Since $u_\theta$ has inverse $u_{-\theta} = \cos \theta - J \, \sin \theta$, it follows that $u_\theta$ is an isometry for all $\theta$ if and only if it is a contraction
for all $\theta$.

A real operator space which is a complex Banach space need not be a complex 
operator space.  Indeed a real $C^*$-algebra which is a complex Banach space need not be a complex 
$C^*$-algebra.  A counterexample is the quaternions, for which it is easy to see 
the $C^*$-algebra statement for example from the fact  that all two dimensional complex 
$C^*$-algebras are commutative.  We argue in a remark below  that the quaternions are not a 
complex 
operator space, although they are obviously a complex Banach space because $\Cdb$ is embedded as a subalgebra. 
In this section we explain why this happens and what is needed to remedy it. 
In particular what more is     needed for  
a real operator space $X$ which is also a complex Banach space 
to be a  complex operator space?   

Sharma defines in \cite[Section 5]{Sharma} (see also  Section \ref{ospm} for more details if desired) 
the real version   $\cM_\ell^{\Rdb}(X)$ of the left operator space multiplier algebra.   We suppress the superscript here 
if the real setting is understood. 
  We have $\cM_\ell(X) \subset CB(X)$ but with a generally different norm.    
      Sharma also defines there the  $C^*$-algebra of `left adjointable multipliers' $\cA_\ell(X)$, which 
      coincides with the diagonal (see Section 1)
      $\Delta(\cM_\ell(X)) = \cM_\ell(X) \cap \cM_\ell(X)^*$ inside  the $C^*$-algebra $p I({\mathcal S}(X)) p$ in the notation 
       in \cite{Sharma} (more details on this are given in 
      Section \ref{ospm} below).
        Similarly one may define the real right multiplier
     algebra and `right adjointable multipliers'  $\cM_r(X)$ and 
     $\cA_r(X)$.   We define the real  operator space centralizer algebra $Z(X)$ (or $Z_{\Rdb}(X)$ when we want to emphasize that this is the real
  case) 
  to be the set $\cA_\ell(X) \cap \cA_r(X)$ in $CB(X)$.   
The  centralizer is a real commutative $C^*$-algebra.
This is because $\cA_\ell(X)$ and $\cA_r(X)$ commute (as may be seen by their canonical representations in the space $I({\mathcal S}(X))$ mentioned
above).  Finally,  $X$ is an operator 
bimodule  over $Z(X)$ in the sense of Theorem 
\ref{5CES} (for basically  the same reason as in the last line; more details on any of this if needed are given in 
      Section \ref{ospm} below).  

\begin{theorem}  \label{chco2}  A  real linear 
map $J : X \to X$ on a real  operator space  corresponds to a complex operator space  structure on $X$ 
(with $\iota \, x = J \, x$) if and only if $$J^2 = -I \; \; \textrm{and} \; \; \| [d_i x_{ij}] \|_n \leq \| [x_{ij}] \|_n \; \; \textrm{and} \; \;  \| [d_j x_{ij}] \|_n \leq \| [x_{ij}] \|_n,$$
 for all $[x_{ij}] \in  M_n(X)$ and $d_1, \cdots, d_n$ maps of form $u_\theta$ as defined in the 2nd paragraph of Section 3 (that 
is, $d_k = u_{\theta_k}$ for real $\theta_k$), and  for all $n \in \Ndb$.  Indeed in the $n > 1$ case of the above displayed inequalities only the 
values $\theta = 0$ and $\frac{\pi}{2}$ are needed in this characterization; that is $d_i =1$, or 
$d_i = J = \iota$).  \end{theorem} 

\begin{proof} The necessity of the condition is obvious.   Conversely, given such a map $J$, from the case $n = 1$  we see that $X$ is a complex Banach space.
We may therefore write $Jx = ix$ henceforth.   We will not need this here but since $u_\theta$ has inverse $u_{-\theta}$ it is easy to see that 
the displayed inequalities are equivalent to $$\| [d_i x_{ij}] \| = \| [d_j x_{ij}] \| = \| [x_{ij}] \|, \qquad [x_{ij}] \in  M_n(X) .$$

From the displayed inequalities 
 it follows that if $x, y \in M_n(X)$ and  $$z = \left[
\begin{array}{l}
x \\
y
\end{array}
\right] ,  \; \; \; \; \; \; w = \left[
\begin{array}{cl}
x & 0 \\
y & 0
\end{array}
\right]$$ then by the case of the displayed inequalities where the
$d_k$ are $i$ or $1$ we have.
$$\left| \left| \left[
\begin{array}{l}
i x \\
y
\end{array}
\right] \right| \right|  =  \left| \left| \left[
\begin{array}{cl}
ix & 0 \\
y & 0 
\end{array}
\right] \right| \right|  \leq \| w \| = \| z \| .$$
Thus by \cite[Theorem 5.4]{Sharma}   we see that the map $Jx = ix$ is a contraction  in $\cM_{\ell}^{\Rdb}(X)$.
A similar argument with $[ x  i \; \; y]$ shows that  $J \in \cM_{r}^{\Rdb}(X)$.  Since $J$ is a contraction there with $J^2 = -I$ it follows 
that $J$ is a unitary in  the `diagonal' of $\cM_{\ell}^{\Rdb}(X))$ mentioned above, with adjoint $-J$.
Thus, $J$ is left adjointable, and similarly it is right adjointable. Thus $J \in Z_{\Rdb}(X)$.    The map $\lambda : \Cdb \to Z_{\Rdb}(X)$ with $\lambda(s + it) = s I + t J$ is 
a homomorphism.   Indeed $\lambda$  is a $*$-homomorphism, so contractive, and hence completely contractive \cite{BT}.   So $X$ is a real operator  
$\Cdb$-bimodule since $X$ is an operator 
$Z_{\Rdb}(X)$-bimodule (see the discussion above Proposition
\ref{3inal} if necessary).  In particular $X$ satisfies Ruan's condition
 characterizing complex operator spaces: $\| \alpha x \beta \| \leq \| \alpha \| \| x \| \| \beta \|$ for  $\alpha, \beta \in M_n(\Cdb)$ and $x \in M_n(X)$. 
 \end{proof}
 
{\bf Remarks.}  1)\ 
The quaternions $\Hdb$ are a complex Banach space and real  operator space 
but not a complex operator space  with the `canonical action' coming from 
the copy of $\Cdb$ as a real subalgebra.      Indeed it  is well known that the complexification of $\Hdb$ is $M_2(\Cdb)$, 
which has no nontrivial (complex or real) ideals.  A nonzero real ideal would have to contain one of the  matrix units 
$e_{ij}$ 
(since $I = e_{11} + e_{22}$ and $x = I x I$), but $M_2 e_{ij} M_2 = M_2$.
If $\Hdb$  was a  complex  operator space then its complexification would be $\Hdb \oplus^\infty \bar{\Hdb}$ (see Lemma \ref{lem2}), which has a nontrivial ideal. 

We make the stronger claim: the real operator space $\Hdb$ has no complex operator space  structure at all, that is there is no action $\Cdb \times \Hdb \to \Hdb$ 
extending the action of the reals and making $\Hdb$  a complex operator space. 
This may be seen since according to the examples above Theorem \ref{centmin}, $Z_{\Rdb}(\Hdb)$ is the center $\Rdb 1$ of $\Hdb$,
whereas if $\Hdb$ has a complex structure then $Z_{\Rdb}(\Hdb)$ would have real dimension $> 1$ by 
Corollary \ref{idee} below, or  by the last assertion in Theorem \ref{comul}. 

The quaternions are however  
a real 
 `operator bimodule',  over the complex numbers, in the sense 
of operator space theory (see Theorem \ref{5CES} above for example).  However unlike the situation for complex operator spaces,
the left or right action of $\Cdb$ is not as `centralizers', and $iw \neq wi$ in general, 
so things are quite subtle and different to what we are used to for complex spaces.

\medskip

2)\  One may rephrase the theorem as: if $X$ is a real operator space   and complex Banach space  then $X$ is a complex operator space 
if and only if $\| d x   \|_n \leq \| x \|_n$  and 
$\| xd \|_n \leq \| x \|_n,$
for all $x \in  M_n(X)$ and diagonal matrices $d$ with entries in $\{ 1, i \}$.

One may also prove part of this theorem using matrix theory to 
verify Ruan's condition in the last line of the proof.  
Any complex matrix may be written via the polar decomposition and diagonalization as a product $U E U'$ for 
complex unitary matrices $U, U'$ and diagonal matrix $E$.  It is known that 
any complex unitary matrix $U$ may be written as 
$K D V$ where $K$ and $V$ are real orthogonal matrices and $D$ is diagonal and unitary (see e.g.\ \cite[Theorem 5.1]{FR}).
Putting these facts together 
gives the theorem (except for its final assertion about $1$ and $i$).  

\begin{corollary}  \label{idee}   A  real operator space  $X$ is a complex operator space 
if and only if it possesses a 
real linear antisymmetry 
$J : X \to X$  (antisymmetry means that $J^2 = -I$)  which is also a real operator space centralizer.  \end{corollary}

It follows that a real operator space  $X$ has a complex operator space structure if and only if its real centralizer algebra $Z(X)$
has a complex Banach space structure, because this happens if and only if $Z(X)$ possesses an antisymmetry (an element with $J^2 = -I$).  We have some interesting 
and at present exciting looking applications of these results
that we have presented in talks and which we hope to present in a future work.    We mention a sample result:

\begin{corollary} \label{ospbid}  Let $X$ be a real operator space such that the real  operator space bidual $X^{**}$ has complex  operator space structure.  
Then $X$ has  complex  operator space structure.  
\end{corollary} 

There is a similar result for the dual in place of the bidual, but it is more complicated to state.  

A second question one may ask about complex structure is whether we can classify all of the scalar multiplications $\Cdb \times X \to X$ on a real operator space that make $X$ a  complex operator space?   
Note that if $X$ and $Y$ are complex operator spaces then $X \oplus^\infty Y$ is a  complex operator space with a `twisted'
complex multiplication given by $i \cdot (x,y) = (ix, -iy)$.  Surprisingly, it turns out that the latter  {\em is the only} example of a complex structure on an
complex operator space space besides the given one:

\begin{theorem}  \label{coosstr}  A  complex operator space $X$  has a unique complex operator space  structure up to complete isometry. 
Moreover for any other complex operator space  structure  on $X$ there exists 
two complex (with respect to the first  complex structure) subspaces  $X_+$ and $X_{-}$ of $X$ with $X = X_+ \oplus^\infty X_{-}$ completely isometrically, 
such that the second complex multiplication of $\lambda \in \Cdb$ with 
$x_+ + x_{-}$ is $(\lambda \, x_+ ) + (\bar{\lambda} \, x_{-})$, for $x_+ \in X_+,  x_{-} \in X_{-}$.  
 \end{theorem} 

\begin{proof} Let $X$ be a complex operator space with a second scalar multiplication making it a complex operator space. 
Let $u : X \to X$ be multiplication by $i$ in the second scalar multiplication.   As we saw in the last theorem,
multiplication by $i$ in the  first product is in $Z_{\Rdb}(X)$, 
and similarly $u \in Z_{\Rdb}(X)$.
Now $Z_{\Rdb}(X)$ is a real commutative $C^*$-algebra.
In such an algebra an element with $f^2 = -1$ 
may be viewed as a function (see 
Example (3) at the end of 
\cite[Section 5.1]{Li}).   
Then $f(x) = \pm i$, so that the involution/adjoint of $f$ is $-f$.   Thus in $Z_{\Rdb}(X)$ we have that $s = -i u$ is selfadjoint with square $I$.  

Therefore $p = \frac{I + s}{2}$ is a projection in $Z_{\Rdb}(X)$.  
Let $X_+ = pX, X_{-} = (1-p)X$.  Claim:  $X = X_+ \oplus^\infty X_{-}$ completely isometrically and as complex operator spaces. 
To see this recall that $p$ commutes with $i$, so is $\Cdb$-linear. By  \cite[Theorem 5.4]{Sharma}
the map $\tau_p$ in \cite[Theorem 4.5.15]{BLM}
is completely contractive on $C_2(X)$.
By  the latter theorem $p$ is a complex left $M$-projection.  Similarly it is a right $M$-projection, and a (complex) complete 
$M$-projection by \cite[Proposition  4.8.4]{BLM}.  This proves the Claim.

Now $u p = is \, \frac{I + s}{2}  = ip$, and similarly $u p^\perp = - i p^\perp$.   We may define a real complete isometry 
$v : X \to X$ with $v(px + (1-p)y) = px - (1-p)y$.  Then $$v (u z) = v(u (px + (1-p)y)) = v(ipx - i(1-p)y) = ipx + i(1-p)y = i z$$ for 
$z = px + (1-p)y$.     \end{proof}

{\bf Remark.}  In particular, if $X$ is the complexification of a real operator space $Y$, then $Z_{\Cdb}(X) = Z_{\Rdb}(X) = Z_{\Rdb}(Y)_c$
by Lemma \ref{mltsco} and Theorem \ref{comul} below.   If $\theta$ is multiplication by $i$ in a different complex operator space
scalar multiplication then by the last theorem we obtain a projection in $Z(X) = Z(Y)_c$.   If $p = q_c$ for projection $q \in Z(Y)$ then we may obtain an $M$-summand decomposition in $Y$  corresponding to the adjusted complex structure as in Theorem \ref{coosstr}.   However 
there may be projections in $Z(X)$ that are not of this form.

\section{Operator space multipliers  and the operator space  centralizer} \label{ospm}

This section is a little more technical and requires a familiarity with the theory of complex operator space multipliers  and  centralizers
(as in e.g.\ \cite{BLM,BZ}).  We will also use the theory of the real  injective envelope $I(X)$ initiated in \cite{Sharma}, and continued in \cite{BCK}. 
As pointed out in \cite{Sharma},  the important construction  4.4.2 in \cite{BLM} relating 
the injective envelope $I(X)$  to the injective envelope $I({\mathcal S}(X))$
of the Paulsen system, is valid in the real case.    
If $p = 1 \oplus 0$ and $q = 0 \oplus 1$ then as usual $I_{11}$ and $I_{22}$ 
are the corners $p I({\mathcal S}(X)) p$ and  $p^\perp I({\mathcal S}(X)) p^\perp$, and $I(X) = p I({\mathcal S}(X)) p^\perp$. 
It is shown in  \cite[Section 4]{BCK} that $I(X)_c = I(X_c)$.   A similar relation holds for the other two important `corners' of $I({\mathcal S}(X))$: 
  
  \begin{lemma} \label{mlts}  For a real operator space $X$ we have $I_{11}(X_c)  \cong (I_{11}(X))_c$ as unital $C^*$-algebras. 
  Similarly $I_{22}(X_c)  \cong (I_{22}(X))_c$.  \end{lemma}
 
   \begin{proof}   Indeed  $I(\cS(X))_c \cong I(\cS(X)_c) \cong I(\cS(X_c))$ as unital $C^*$-algebras, via a $*$-homomorphism that preserves the projection $p = 1 \oplus 0
   \in \cS(X)$.  Then  
  $$I_{11}(X_c) = p I(\cS(X_c)) p \cong  p I(\cS(X))_c p \cong (p I(\cS(X)) p)_c = (I_{11}(X))_c.$$    Similarly, $I_{22}(X_c)  \cong (I_{22}(X))_c$.   
   \end{proof}  

We defined the left and right real operator space multiplier algebras $\cM_\ell(X)$ and $\cM_r(X)$ just above Theorem \ref{chco2}, as well as the 
real $C^*$-algebras of left and right adjointable multipliers $\cA_\ell(X)$ and $\cA_r(X)$, and the real  operator space centralizer algebra $Z(X)$. 
Giving a little more detail from \cite{Sharma}, 
we give  
  matrix norms as in the complex case (see 4.5.3 in \cite{BLM}): i.e.\ so that  $M_n(\cM_\ell(X)) \cong \cM_\ell(C_n(X))$ isometrically.
  Also $T \in {\rm Ball}(\cM_\ell(X))$ if and only if  there exists a 
  real linear complete isometry $j : X \to B(H)$ for a real Hilbert space $H$ and an operator $S \in {\rm Ball}(B(H))$ with $j(Tx) = Sj(x)$ for all $x \in X$.   
 Complexifying this,  if $T \in \cM_\ell(X)$ we see that  $T_c \in \cM_\ell(X_c)$.  The most useful characterization however is the $\tau_T$ criterion
 which in the real case is \cite[Theorem 5.4]{Sharma}. 
  (Sharma does not state this, but  the real versions of all statements in 4.5.1--4.5.9, and 4.5.12--4.5.13 from \cite{BLM} are valid by the same arguments,
  with the possible exception of Theorem 4.5.2 (v).    Items 4.5.14-4.5.15 are in \cite{Sharma}, and we will check 4.5.10 and 4.5.11 below.   
  We will discuss Banach-Stone theorems such as 4.5.13 briefly in Section \ref{inj4}.) 
  
   As in the complex case we obtain a completely contractive one-to-one homomorphism $\cM_\ell(X) \to CB(X)$, and a completely isometric 
  homomorphism $\cM_\ell(X) \to I_{11}(X)$.   Indeed 
  we have $$\cM_\ell(X) \; \cong \; \{ a \in I_{11}(X) : a j(X) \subset j(X) \} ,$$ where $j : X \to I(X)$
  is the canonical inclusion (that is, $(I(X) , j)$ is an injective envelope of $X$).   
      
  \begin{lemma} \label{mltsco}  For a real operator space $X$ we have $\cM_\ell(X_c) \cong \cM_\ell(X)_c$ completely isometrically as operator algebras, and 
  $T \mapsto T_c$ is the canonical map $\cM_\ell(X)$ into  
  this complexification.  Similarly for the spaces of right multipliers; and 
   $\cA_\ell(X_c) \cong \cA_\ell(X)_c$  $*$-isomorphically.  
   The operator space centralizer algebra $Z(X)$ is a commutative unital real $C^*$-algebra, and $Z(X_c) \cong Z(X)_c$ $*$-isomorphically.
 \end{lemma}
 
   \begin{proof}  
  We may identify 
  $\cM_\ell(X_c)$ with $$\{ a \in I_{11}(X_c) : a j_c(X_c) \subset j_c(X_c) \} \cong \{ b + i c \in (I_{11}(X))_c : ( b + i c)  j_c(X_c) \subset j_c(X_c) \}  .$$
  The latter inclusion is equivalent to $b$ and $c$ being in $\cM_\ell(X)$.    It follows (if necessary thinking of $b + ic$ as the usual $2 \times 2$ matrix as in (\ref{ofr})), 
  that the right side of the last displayed equation is $\cM_\ell(X)_c$.  That is, $\cM_\ell(X_c) \cong \cM_\ell(X)_c$ completely isometrically as operator algebras. 
  It is easy to see the assertion regarding $T_c$: if $Tj(x) = aj(x)$ for $a \in I_{11}(X)$ then  
  $T_c \, (j_c)(x+iy) = a j_c(x)+i a j_c(y)$.
   Similar arguments hold in the $\cM_r(X)$ case.
  
 Using the fact $\Delta(A_c) = \Delta(A)_c$ from the introduction concerning 
  the diagonal operator algebra, we have  $$\cA_\ell(X)_c = \Delta(\cM_\ell(X))_c \cong   \Delta(\cM_\ell(X)_c) \cong  \Delta(\cM_\ell(X_c)) = \cA_\ell(X_c),$$  and 
  $$Z(X_c) = \cA_\ell(X_c) \cap \cA_r(X_c) = \cA_\ell(X)_c \cap \cA_r(X)_c = Z(X)_c .$$ 
Moreover for $T \in Z(X)$ the involution of $T$ in $\cA_\ell(X)$ coincides with its involution in $\cA_r(X)$, 
since the analogous statement is true  in $Z(X_c)$ \cite{BZ}.
So $Z(X)$ is a real $C^*$-algebra.   \end{proof}

We have  $\cM_\ell(X) = \{ T \in B(X) : T_c \in \cM_\ell(X_c) \}$.  Indeed for example clearly $\tau_T$ is completely contractive if and only if  $\tau_{T_c} = (\tau_T)_c$  is completely contractive.  The map $T \mapsto T_c$ restricts to a faithful
contractive homomorphism, hence $*$-homomorphism, $\cA_\ell(X) =  \Delta(\cM_\ell(X)) \to  \cA_\ell(X_c) = \Delta(\cM_\ell(X))$.
Thus  if $T \in \cA_\ell(X)$ then $(T_c)^* = (T^*)_c$.   
It is easy to see that $\cA_\ell(X) = \{ T \in B(X) : T_c \in \cA_\ell(X_c) 
\}$, and a  similar relation holds  for  $Z(X)$. 
Indeed   $T_c \in Z(X_c) = \cA_\ell(X_c) \cap \cA_r(X_c)$ iff  $T \in \cA_\ell(X) \cap \cA_r(X)$, and iff $T \in Z(X)$.

As in the complex case,  $X$ is an operator bimodule (or $h$-bimodule) over all of these spaces in the sense of \cite[Section 3.1]{BLM}. 
 Indeed $X$ is 
  an operator  $\cM_\ell(X)$-$\cM_\ell(X)$-bimodule, 
  which in turn can be viewed as an $\cM_\ell(X)$-$\cM_\ell(X)$-sub-bimodule of $X_c$.   The actions of $\cA_\ell(X), \cA_r(X)$ and $Z(X)$ on $X$ 
are restrictions  of the  $\cM_\ell(X)$-$\cM_\ell(X)$-bimodule action.  
  These bimodules can as usual  all be concretely viewed at any time 
   as subspaces of a $C^*$-algebra, e.g.\ of the $C^*$-algebra  which is the injective envelope of the Paulsen system.   For example $X$ regarded as 
  an operator  $\cM_\ell(X)$-$\cM_\ell(X)$-bimodule may be viewed in the $1$-$2$ corner of $I({\mathcal S}(X))$, with $\cM_\ell(X)$ a subalgebra
  of the $1$-$1$ corner $I_{11}$ as above.  The action $T(x)$, for $T \in \cM_\ell(X)$ and $x \in X$, 
  is identified in the usual way with the product of the two $2 \times 2$ matrices in $I({\mathcal S}(X))$
  corresponding to the copy of $T$ in $I_{11}$ and the copy of $X$ in the $1$-$2$ corner.   
  There is a subtlety however with the $Z(X)$-bimodule action. 
  We often write $a x = x a$ for $a \in Z(X), x \in X$ and this is perfectly true if interpreted correctly as an abstract bimodule action. 
 However  if $X$ is represented concretely, such as above in the $1$-$2$ corner of $I(\cS(X))$, one should usually rather write 
   $\rho(a) x = x \pi(a)$, where $\pi$ and $\rho$ are   faithful $*$-homomorphisms with domain $Z(X)$.    
  That is, although the right action of $Z(X)$ on $X$ coincides with the left action, if one represents $X$ as a concrete bimodule as above say, $Z(X)$ usually gets represented in `two different places'.

\begin{proposition}
\label{3inal}  Let  $X$ be a real  operator space.
\begin{enumerate}
\item [{\rm (1)}]  A real linear map $u  \colon X \rightarrow X$ is in
$\A_l(X)$ if and only if there exist a real Hilbert space $H$,
a  linear complete isometry $\sigma  \colon X
\rightarrow B(H)$, and a map $v : X \to X$ with $\sigma(ux)^* \sigma(y)  = \sigma(x)^* \sigma(vy)$ for all $x, y \in
X$. 
\item [{\rm (2)}]  Same as {\rm (1)},  but with the conditions involving $v$ replaced by 
existence of $S \in B(H)$ with 
$\sigma(ux) = S \sigma(x)$ for all $x \in
X$, and such that also $S^* \sigma(X) \subset \sigma(X)$. 
\item [{\rm (3)}]  $\A_l(X) \cong \{ a \in I_{11}(X) : a X \subset X
\; \; \text{and} \; \; a^* X \subset X \}$
as $C^*$-algebras.
\item[{\rm (4)}]  If $u, S$ and $\sigma$ are as in
{\rm (1)} (resp.\ (2)), then
the involution $u^*$ in $\A_l(X)$ is $v$ 
(resp.\ the map $x \mapsto \sigma^{-1}(S^* \sigma(x))$ on
$X$).
\item [{\rm (5)}] The canonical inclusion
map from  $\A_l(X)$ into $CB(X)$
is a completely isometric homomorphism.  Moreover $\| T \|_{cb} = \| T_c \| = \| T \|$ for $T \in \A_l(X)$, 
so the inclusion viewed as a 
map  into $B(X)$ is an isometric homomorphism.
\end{enumerate}
\end{proposition}

\begin{proof}  If $S \in B(H)$ exists with all the properties in (2)  then $v = \sigma^{-1}(S^* \sigma(\cdot))$ satisfies the condition in (1). 
Conversely if $v$ satisfies the condition in (1) then $v_c : X_c \to X_c$ satisfies
$\sigma(u_c x)^* \sigma(y)  = \sigma(x)^* \sigma(v_c y)$ for all $x, y \in
X_c$. By \cite[Theorem 8.4.4]{BLM} we deduce that $u_c \in \A_l(X_c)$ so that $u \in \A_l(X)$ by an earlier observation.
   The rest of the proofs of  (2)--(4) are as in the complex case and using the real cases discussed earlier    of results cited in   \cite[Proposition 4.5.8]{BLM}.    
 
   For (5), if $\A_l(X)$ we have  $$\| T \|_{\cM_\ell(X))} = \| T_c \|_{\cM_\ell(X_c))}  = \| T_c \|_{\rm cb} = \| T \|_{\rm cb}.$$  From the complex case 
this equals $\| T_c \|$, by the same proof but using Theorem \ref{cid} (the real version of \cite[Theorem A.5.9]{BLM}).   The matricial case follows similarly, as in the complex case. 
The last assertion follows from Theorem \ref{cid}. 
 \end{proof}

Note that $Z_{\Rdb}(X)$ is a real commutative $C^*$-algebra but need not be a real $C(K)$ space (a simple example is $X = \Cdb$ considered as a real  operator space). 

\bigskip

{\bf Examples.} \begin{itemize} 
\item [(1)]  For any real approximately unital operator algebra $A$ we have that $\cM_\ell(A) = LM(A)$ and $\cM_r(A) = RM(A)$, just as in the proof of \cite[Proposition 4.5.11]{BLM}.  Hence  
$\cA_\ell(A) = \Delta(\cM_\ell(A)) = \Delta(LM(A))$.  If $B$ is a real $C^*$-algebra then it is then easy to deduce  
that $\cA_\ell(B)$  is just $M(B)$.  
We claim  that $Z(A)$ is the diagonal of the center of $M(A)$.  Indeed, if $T \in Z(A)$ then $T_c \in Z_{\Cdb}(A_c)$, so that 
$T_c$ is in the diagonal of the center of $M(A_c)$.  Thus $T \in M(A)$ by our earlier discussion of the multiplier algebra of $A$.
It is easy to see that $T$ is in the  center of $M(A)$.  The map $Z(A)$ into the center of $M(A)$ is a contractive homomorphism, so maps into 
 the diagonal of the center of $M(A)$.  Conversely, if $T$ is in  the diagonal of the center of $M(A)$ then $T \in \Delta(LM(A)) = \cA_\ell(A)$.
 Similarly $T \in  \cA_r(A)$, so that $T \in Z(A)$. 

  \item [(2)] If $X = H_{\rm col}$, a real Hilbert column space, then as in the complex case 
$\cM_\ell(X) = \cA_\ell(X) = B(H)$, $\cM_r(X) = \Rdb I$, so that $Z(X) = \Rdb I$.  

\item [(3)]  Suppose that  $X$ is a a minimal real operator space, that is $X  = {\rm Min}(E)$ for a real Banach space $E$ (see \cite{Sharma}).  Then 
it is known that the Banach space injective envelope is a $C(K)$ space, and 
\cite[Proposition 4.10]{Sharma} shows that the injective envelope of $X$ is a $C(K)$ space with its canonical operator space structure.
By the argument in \cite[4.5.10]{BLM} we have that $\cM_\ell(X)$ is a minimal operator space, and $\cA_\ell(X)$ is a commutative $C^*$-algebra.
Indeed we have the following, which  is the real case of
 \cite[Proposition 7.6 (i)]{BZ}.   \end{itemize} 
 
 \begin{theorem} \label{centmin} For a real Banach space $X$ we have $$\cM_\ell({\rm Min}(X)) = \cA_\ell({\rm Min}(X)) = Z({\rm Min}(X)) = {\rm Cent}(X) = \cM(X)$$
 where $ {\rm Cent}$ and $\cM$ are the real Banach space centralizer and multiplier algebra \cite{Behr,HWW}.  This also agrees 
 with $\{ T \in B(X) : T_c \in \cM(X_c) \}$, and for such $T$ we have $\| T \| = \| T_c \|$. 
\end{theorem}

\begin{proof}  It is known (and obvious from the definitions)  that ${\rm Cent}(X) = \cM(X)$ in the real case \cite{Behr,HWW}. 
We have $T \in  \cM_\ell({\rm Min}(X))$ if and only if  $$T_c \in \cM_\ell(({\rm Min}(X))_c) = \cM_\ell({\rm Min}(X_c)) = \cM(X_c),$$ 
using \cite[Proposition 2.1]{Sharma} and 4.5.10 in \cite{BLM}. 
If  $T \in \cM(X)$ then
  as in \cite[Theorem 2.1]{BShil} one sees that there is a compact space $K$, a linear isometry $\sigma : X \to C(K),$ and
an $f \in C(K),$ such that $\sigma(Tx)= f \sigma(x)$ for $x \in X$.   
By a universal property  of Min (cf.\ \cite{Sharma}) it is easy to see $\sigma$ is a complete isometry on Min$(X)$.
So $T$ is in $\cA_\ell({\rm Min}(X))$ (since $f$ is selfadjoint) and indeed in $Z({\rm Min}(X))$ since the $C^*$-algebras concerned 
all commute.   Thus we have shown that 
$${\rm Cent}(X) = \cM(X) \subset Z({\rm Min}(X)) \subset \cA_\ell({\rm Min}(X)) \subset \cM_\ell({\rm Min}(X)) .$$ 
Now $I({\rm Min}(X)) = (C(K,\bR),j)$, a real $C(K)$ space, as stated above the theorem. 
We have from the real version of \cite[Proposition 4.4.13]{BLM} in conjunction with the representation of $\cM_\ell$ above Lemma \ref{mltsco}, that 
$\cM_\ell({\rm Min}(X)) = \{ f \in C(K) : f j(X) \subset j(X) \}$. By the real version of \cite[Theorem 3.7.2]{BLM} 
we see that $\cM_\ell({\rm Min}(X)) \subset \cM(X)$.

That  $\| T \| = \| T_c \|$ follows from the above and e.g.\.  Proposition \ref{3inal} (5).
\end{proof}

{\bf Remark.}   In \cite{Behr,HWW} it is shown that for a complex Banach space $X$ we have ${\rm Cent}_{\Cdb}(X) = {\rm Cent}_{\Rdb}(X) + i \, {\rm Cent}_{\Rdb}(X)$.  

\begin{theorem} \label{comul}  If $X$ is a complex operator space viewed as a real operator space $X_r$ we have 
  $\cM_\ell(X_r) = \cM_\ell^{\Cdb}(X)$, $\cA_\ell(X_r) = \cA_\ell^{\Cdb}(X)$,
     and $Z_{\Rdb}(X_r) = Z_{\Cdb}(X)$.  The projections in the latter algebra are the 
   (complex) complete $M$-projections (of e.g.\ {\rm \cite[Section 7]{BZ}} or \cite[Chapter 4]{BLM}).  
\end{theorem}

\begin{proof}  If $u$ is multiplication by $i$ on $X$ then $u \in {\rm Ball}(\cM_\ell(X_r))$,
and indeed is in the centralizer as we saw in the proofs in Section \ref{csiros}.
Hence $u$ commutes with all members of $\cM_\ell(X_r)$. 
It follows that operators in $\cM_\ell(X_r)$ are $\Cdb$-linear on $X$, and therefore are in  $\cM_\ell^{\Cdb}(X)$.
Thus $\cM_\ell(X_r) = \cM_\ell^{\Cdb}(X)$.   This is a completely isometric 
identification since $C_n^{\Rdb}(X_r) \cong C_n^{\Cdb}(X)$ real completely isometrically and 
thus $$M_n(\cM_\ell(X_r)) =  \cM_\ell(C_n(X_r)) =  \cM_\ell^{\Cdb}(C_n(X)) \cong M_n(\cM_\ell^{\Cdb}(X))$$ 
  isometrically.  
Taking the `diagonal', we have  $\cA_\ell(X_r) = \cA_\ell^{\Cdb}(X)$.
We used the fact here that for a unital subalgebra $A$ of $B(H)$ in the complex setting, if we view 
  $B(H)$ as a real $C^*$-algebra $B(H)_r$ then the `real and complex diagonal algebras' of $A$ coincide.
Then $$Z_{\Rdb}(X_r) = \cA_\ell(X_r) \cap \cA_r(X_r) = \cA_\ell^{\Cdb}(X) \cap \cA_r^{\Cdb}(X) = Z_{\Cdb}(X)$$ as desired.
The last assertion may be found in e.g.\ \cite[Section 7]{BZ}. 
\end{proof}

{\bf Remarks.}   1)\ The above is  quite different to the Banach space case, where ${\rm Cent}_{\Rdb}(X)$ need not contain ${\rm Cent}_{\Cdb}(X)$.   Indeed
  if we take $X = \Cdb$ then ${\rm Cent}_{\Fdb}(X) = \Fdb \cdot I$, while $Z(X) = \Cdb = Z_{\Rdb}(X_r) = \cA_\ell(X)$. 
  Nor do we have $Z_{\Rdb}(X) \subset {\rm Cent}_{\Rdb}(X)$, where the latter is the Banach space centralizer algebra, in contrast to the complex theory
   \cite[Corollary 7.2]{BZ}.  Indeed $Z_{\Rdb}(A)$ is not a subset of ${\rm Cent}_{\Rdb}(A)$ in general even for real commutative two dimensional $C^*$-algebras.
      Nor 
   is ${\rm Cent}_{\Rdb}(X) \subset Z_{\Rdb}(X)$ in general.  Indeed  suppose that $X$ is 
   a complex operator space with a complex $M$-projection $P$ which is not a complete $M$-projection.
   Then $P$ is a real $M$-projection  
   so is in  ${\rm Cent}_{\Rdb}(X)$.   
   However $P$ is not in $Z_{\Rdb}(X)$, for if it were then one sees that $P^{(n)}$ is an $M$-projection for all $n$, contradicting that 
   $P$ is not a complete $M$-projection.  Sharma does not discuss real complete $M$-projections, but just as in the complex case one sees that these 
   are just the $P \in B(X)$ with $P^{(n)}$ is an $M$-projection for all $n$. 
   
   We recall that a real $M$-projection on a complex Banach space is
  a complex $M$-projection \cite[Theorem I.1.23]{HWW}.  Similarly, a real complete 
  $M$-projection $P$ on a complex operator  space is
  a complex complete $M$-projection.  This may be seen easily by the proof of the Claim in Theorem  \ref{coosstr}. 

\smallskip 
  
2)\ If $X$ is a complex operator space then we recall from \cite[Section 4]{BCK} that $I^{\Rdb}(X) = I^{\Cdb}(X)$.
Replacing $X$ by $I(X)$ in the last result
we see that $I_{11}^{\Rdb}(X) = \cM_\ell(I(X)_r) = \cM_\ell^{\Cdb}(I(X)) = I_{11}^{\Cdb}(X)$. Similarly for $I_{22}$. 
  
  \begin{proposition}  \label{mlbid} We have  $T \in \cM_\ell(X)$ if and only if  $T^{**} \in \cM_\ell(X^{**})$, and the multiplier norms of these coincide. 
  Similarly $T \in Z(X)$ if and only if  $T^{**} \in Z(X^{**})$.  Also 
 $\cA_\ell(X) \subset \cA_\ell(X^{**})$ isometrically  via  a faithful $*$-homomorphism.   
  \end{proposition}

\begin{proof}  The first statement follows  by the argument  in  \cite[Proposition 5.14]{BZ}.   Similarly for the second statement. 
For $Z(X)$ we have $T \in Z(X)$ iff $T_c \in Z(X_c)$, and  iff $(T_c)^{**}  \in Z((X_c)^{**})$ by \cite[Theorem 7.4]{BZ}.
 Since $(T_c)^{**}  = (T^{**})_c$ and $(X_c)^{**} = (X^{**})_c$, the latter happens iff $(T^{**})_c \in Z((X^{**})_c)$, which as before is
 equivalent to $T^{**} \in Z(X^{**})$. \end{proof}

{\bf Remark.}  We do not in general have the `if and only if' above
 in the case of  $\cA_\ell(X)$.  That is one may have $T \in  \cM_\ell(X) \setminus \cA_\ell(X)$ with $T^{**} \in \cA_\ell(X^{**})$.

 The remaining 
  results in Section 4.6 in \cite{BLM}  seem unchanged in the real case.  In Section 4.7, the first result follows by complexification, and the rest of the results
  follow with unchanged proofs.  As in the complex case, $\cM_{\ell}(X)$ is a dual operator algebra, and its subspace $\cA_{\ell}(X)$ is a real von Neumann subalgebra, if $X$ is a dual operator space.  
  It follows that $Z(X)$ is a commutative real von Neumann subalgebra, whose projections as in the complex case are the complete $M$-projections.  
  Indeed $Z(X)$ may be viewed as the fixed points of the weak* continuous period 2 $*$-automorphism of the 
  complex commutative von Neumann algebra $Z(X)_c \cong Z(X_c)$.
  Most of 
Section 4.8 is covered in \cite{Sharma}, with the notable exception of complete (two-sided) $M$-ideals.
We have already mentioned in Remark 1 after Theorem \ref{comul}  that  one sees  just as in the complex case that complete $M$-projections 
   are just the $P \in B(X)$ with $P^{(n)}$  an $M$-projection for all $n \in \Ndb$.    
   As in the complex case these are the left $M$-projections that are also right $M$-projections, and they are precisely the 
   projections in $Z(X)$ as in the complex case.  
 
 Note that if a real operator algebra $A$ has both a left and a right cai
   then it has a cai as in e.g.\ \cite[Proposition 2.5.8]{BLM}.

\begin{theorem} \label{cez}  If $A$ is an approximately unital real operator algebra then
the complete $M$-ideals (resp.\ complete $M$-summands) are 
exactly the two-sided ideals in $A$ with contractive approximate identity
(resp.\ $Ae$ for a central projection $e \in M(A)$). 
\end{theorem}  

\begin{proof} Let $P$ be a complete  $M$-projection on an approximately unital operator algebra.
So it is a left $M$-projection and a right 
$M$-projection.  
Thus $P_c$  is a  complex left $M$-projection on $A_c$  by \cite[Proposition 5.8]{Sharma}, and similarly on the right,
hence it  is a complex  complete  $M$-projection.    By the complex case in Theorem 4.8.5 from \cite{BLM} we have that 
$P_c(x) = ex$   for a central projection $e \in M(A_c)$.  Since $M(A_c) = M(A)_c$ by Lemma \ref{mulr},
and $P_c A = e A \subset A$, it follows that $e \in M(A)$.   This proves the statement about $M$-summands.  
Similarly it follows as in \cite[Corollary 5.9]{Sharma}
   that a real subspace $J$ is a complete $M$-ideal (or summand) in $X$ if and only if  $J_c$ is  is a complete $M$-ideal  (or summand) in $X_c$.
   From this the statement about $M$-ideals follows, or it may be deduced from \cite[Corollary 5.11]{Sharma}.
   \end{proof} 

It is not at all obvious if the word `complete' in the last result can be dropped, indeed such a result has not previously appeared in 
the literature even in the real $C^*$-algebra case.   For a commutative real $C^*$-algebra $A$ the $M$-ideals are the closed ideals.  Indeed by \cite[Theorem 3.1]{BCK} $A^{**}$ is the 
$L^\infty$ sum of real and complex $C(K)$ spaces.  As we said earlier real and complex $M$-projections on a complex space coincide. 
Any real $M$-summand $V$ of  $A^{**}$ corresponds to 
real $M$-summands of these $C(K)$ spaces. 
Thus $V = e A^{**}$ for an orthogonal projection $e \in A^{**}$. 
Recently in joint work with Matt Neal, A.\ Peralta and S.\ Su we have proved  that the $M$-ideals in  a general real $C^*$-algebra (and in more general objects of interest) are the closed ideals, and coincide with the complete $M$-ideals. 
 See the Acknowledgements below. In particular, we have fully generalized Theorem 4.8.5 of \cite{BLM}  to the real case.
  We remark that \cite[Theorem 4.8.5 (4)]{BLM} was later improved to replace LM by the multiplier algebra $M(A)$, and this holds also in the real (even Jordan) operator algebra case as will be discussed in a future paper (see the Acknowledgements below).
   
 \section{Operator spaces} \label{os}
 
 The remaining sections of our paper, as stated in the introduction, verify in a very economical, linear,
 and somewhat systematic format
the real case of the remaining theory in several chapters of \cite{BLM}, and establish how the basic constructions there 
interact  with the complexification. 
We begin with Chapter 1 of that text. 
 A few of the facts below appear in the papers of Ruan and Sharma (see e.g.\ the list at the end of Section 2 of \cite{Sharma}). 
 Ruan showed in \cite{RComp} that 
$(X_c)^* \cong (X^*)_c$ completely  isometrically.   For linear functionals we have 
$\| \varphi \| = \| \varphi \|_{cb}$, as follows from \cite[Lemma 5.2]{ROnr}.
  It is pointed out however in \cite[Proposition 2.8]{Sharma} that $(X_r)^* \neq (X^*)_r$  completely isometrically
for a complex operator space 
(the mistakes in the proof of that Proposition are easily fixed).    Here $X_r$ is the space regarded as a real operator space.

We note that a real $C^*$-algebra $A$ has a unique operator space structure such that $M_n(A)$ is a real $C^*$-algebra for all $n \in \Ndb$.   This is
because a real $*$-algebra has at most one $C^*$-algebra norm, as follows immediately 
by complexification, or from e.g.\ \cite[Lemma 2.2]{Sharma}.   Proposition 1.2.4 in \cite{BLM} then follows from \cite[Theorem 2.6]{BT}.
The assertion there that $\pi$ is a complete quotient map with closed range can easily be seen by considering $\pi_c$. 
 It is interesting however that the Russo-Dye theorem holds for real von Neumann algebras \cite{MMPS}; the Kadison-Schwartz inequality is discussed there too. 
 On quotient operator spaces we have $X_c/Y_c \cong (X/Y)_c$ by Proposition \ref{cbcmpx} (the proof of this in \cite{Sharma} may be corrected as   in \cite[Proposition 5.5]{BK}). 
   It is easy to see that the  operator space complexification commutes with 
$c_0$ and $\ell^\infty$ sums 
of operator spaces (see e.g.\  the proof of \cite[Lemma 5.14]{Sharma}).   The rest of Section 1.2 in \cite{BLM} is in the small real operator space literature, or follows as in the complex case or (as in the case of e.g.\ 1.2.31) by complexification. 
 The exception is 1.2.30, for obvious reasons; we will not discuss interpolation in the present paper. 

In what follows $\bar{X}$ is the `conjugate operator space’ of a complex operator space $X$.  That is  $\bar{X}$ is
the set of formal symbols $\bar{x}$ for $x \in X$, 
with matrix norms $\| [ \bar{x}_{ij} ] \| = \| [x_{ij}] \|$.    With `conjugate scalar multiplication' $\bar{X}$ is a complex operator space; the map $x \mapsto \bar{x}$ is 
conjugate linear.  (See Proposition 2.1 and the Remark after it
in e.g.\ \cite{BWinv}.)   This defines a functor on the complex operator space, with the action on morphisms being $\bar{T}(\bar{x}) = \overline{T(x)}$. 
In practice this construction $\bar{X}$ is  often extremely useful.   The first reason for this is in checking certain facts about a complexification $Y_c$ we 
very often must use the important conjugate linear isomorphism $\theta_Y$ on $W = Y_c$.  One often wants to apply results 
from the complex theory to $\theta_Y$, and to do this it is often convenient to view conjugate linear  maps into $W$ as complex linear maps into $\bar{W}$. 
E.g.\ Lemma \ref{oone} is an example of this. 
The next result is also useful for similar reasons (see \cite{Cecco} for an example of this).

 \begin{lemma} \label{lem2}  The  operator space complexification of a complex operator space $X$ is complex linearly completely isometric to $X \oplus^\infty \bar{X}$, where we 
identify $X$ with $\{ (x,\bar{x}) : x \in X \}$, and identify $x + iy$ in the complexification with $( x + iy , \overline{x - iy}) \in X \oplus^\infty \bar{X}$.  
If $B$ is a complex $C^*$-algebra then  $B_c \cong B \oplus^\infty  B^{\circ}$ $*$-isomorphically (as complex $C^*$-algebras).
Here $B^\circ$ is the `opposite' $C^*$-algebra of $B$.
  \end{lemma} 

\begin{proof}    Note that $X + i X$ becomes  $\{ (x,\bar{x})  + (iy ,-\overline{iy}) :  x , y \in X \} =  X \oplus_\infty \bar{X}$.
Also, $(x,\bar{x})  = (iy ,-\overline{iy})$ implies that $x = iy = -i y$, so that $x = y = 0$ and $X \cap (i X) = (0)$.  
The embedding $x \mapsto (x,\bar{x})$ is clearly completely isometric (by definition of $\bar{X}$ above).   Finally if we define $\theta( x + iy , \overline{x - iy}) = (x - iy , \overline{x + iy})$, then 
$\theta((a, \bar{b})) = (b ,\bar{a})$ for $a, b \in X$, which is clearly completely isometric, and the copy of $X$ is the set of fixed points of $\theta$.   So by Ruan's uniqueness theorem,
the operator space complexification of  $X$ is $X \oplus^\infty \bar{X}$.  

If $B$ is a complex $C^*$-algebra then $\bar{B}$ is a $C^*$-algebra with product 
$\bar{x} \bar{y} = \overline{xy}$, indeed $\bar{x} \mapsto (x^*)^\circ$  is a complex linear $*$-isomorphism $\bar{B} \to B^\circ$.  \end{proof}

{\bf Remark.}  We remark that in the last result  $X \oplus^\infty \bar{X}$ is real completely isometric to the set of matrices  of the form 
in equation (\ref{ofr}) with $x, y \in X$.  However the latter set is not always complex linearly completely isometric to $X_c$ (similarly to the fact mentioned in Section 2 that 
a Banach space $E$ need not be complex isometric to $\bar{E}$),  so it is not very useful for many purposes.   Note also that the statement in the lemma about $C^*$-algebras does not imply that $A \oplus A^{\circ}$ is completely isometric to $A_c$ for a (complex) operator algebra $A$, where $A^\circ$ is the opposite algebra of 2.2.8 in \cite{BLM}.

\medskip

The  basics of the theory of real operator systems are much the same as in the complex case (see  \cite{RComp,ROnr, Sharma,BT}) in many ways, most results
being proved in the same way as the complex case, or simply following from that case by complexification. 
We have already indicated some of the problematic issues with real operator systems and states in other papers (see also \cite{RComp} and \cite{BT}), such as 
the selfadjoint or positive elements elements not necessarily spanning a real operator system  ${\mathcal S}$, or they may be all
of ${\mathcal S}$.    See also comments on \cite[Section 2.1]{BRS} below.

 The facts in Section 1.4 of \cite{BLM} are valid with the same proofs.  
   See e.g.\ Proposition 2.4 and the Remark on page 104 of \cite{Sharma} for some of this.     
   
 \begin{lemma} \label{comw}  Let $X$ be a real operator space,  which is the operator space dual of 
 another real operator space $Y$.  Then  $X$  is completely isometrically isomorphic, via a homeomorphism for the weak* topologies, to a weak* 
 closed subspace of $B(H)$ for a real Hilbert space $H$.  Also, $X$ is a weak* closed real subspace of its operator space complexification $X_c \cong (Y_c)^*$, and 
 the canonical  maps $X \to X_c$ and $X_c \to X$ and $\theta_X : X_c \to X_c$
 are weak* continuous.  
 
 Conversely, any weak* closed subspace $X$ of $B(H)$, is the operator space dual of $B(H)_*/X_\perp$.  \end{lemma}

  \begin{proof}  The first  and last statement follows most quickly perhaps from the proof of
  \cite[Lemmas 1.4.6 and  1.4.7]{BLM}.   If $X$ is a  weak* closed subspace $X$ of $B(H)$ then in
  the setup described   early in Section 2 of 
   \cite{BCK}, $B(H)_c$ may be viewed as a weak* closed real subspace of $M_2(B(H)) \cong B(H^{(2)}) \cong B_{\Rdb}(H_c)$, and the induced embedding $B(H)_c \hookrightarrow B_{\Rdb}(H_c)$ has range inside $B_{\Cdb}(H_c)$, and coincides with 
 the canonical map $B(H)_c \to B_{\Cdb}(H_c)$.   The latter range is weak* closed.
 The isomorphism  $B(H^{(2)}) \cong B_{\Rdb}(H_c)$ is a weak* homeomorphism.  It follows that 
 the canonical  maps $B(H) \to B(H)_c$ and $B(H)_c \to B(H)$ (`real' and `imaginary' part) are weak* continuous. 
 Indeed $S_t + i T_t \to S + i T$ weak* iff $S_t \to S$ weak* and $T_t \to T$ weak*. 
 It follows that 
 we have $X_c$ weak* closed in $B(H)_c \cong B(H_c)$, and 
   the canonical  maps $X \to X_c$ and $X_c \to X$ are weak* continuous, as thus is $\theta_X$ on $X_c$.  
   Thus $X$  is a weak* closed real subspace of its complexification. 
   
   If $X$ has a real operator space predual $Y$ then $X_c$ has  operator space predual $Y_c$ of course.
   By the net convergence formulation above, the associated weak* topology on $X_c$
   coincides with the one inherited from $B(H_c)$ above if the  embedding $X \subset B(H)$ is weak* continuous with respect to $\sigma(X,Y)$,
   such as the one in the first statement.    \end{proof} 
   
   A real $C^*$-algebra is representable as a real von Neumann algebra if and only if it has a Banach space predual
 \cite{IP}.   
 Such Banach space predual is unique by 6.2.5 in the same text. 
   By the last part of \cite[Section 5.5]{Li} the bidual of a real $C^*$-algebra $A$ is a real $W^*$-algebra, and 
 $(A_c)^{**} \cong (A^{**})_c$.   Indeed the
 theory of the real bidual of an operator space or operator algebra  is very similar to the complex variant (e.g.\ in \cite[Sections 1.4 and 2.5]{BLM}).
 
 \begin{lemma} \label{oone} The operator space 1-direct sum $W = \oplus^1_\alpha \, (X_\alpha)_c$ is the reasonable complexification of  $V = \oplus^1_\alpha \, X_\alpha$.   
 \end{lemma}
 
   \begin{proof}   
     The  inclusion $X_\alpha \to W$ induces a complete contraction $\kappa : V \to W$, by the real variant of
 1.3.14 in \cite{BLM}.   The universal property of the 1-sum applied to the complexification of the inclusion $i_\alpha :  X_\alpha \to V$
 gives a complete contraction $r : W \to V_c$ with $r \circ \kappa = I_V$.  So 
  $\kappa$ is a complete isometry.   It is not hard to show that  $\bar{W}$ has the universal property of $\oplus^1_\alpha \, \overline{(X_\alpha)_c}$,
 so that these operator spaces may be identified.
 The composition of  $\theta_{X_\alpha}$ with the canonical map $(X_\alpha)_c \to W \to \bar{W}$ 
 induces  a linear complete contraction $u : W \to \bar{W}$.  Similarly we obtain a map  $v: \bar{W} \to W$,
with $v u = I_W$.   Then $\overline{u(\cdot)}$ is a period 2 conjugate linear 
 complete isometric surjection $\theta : W \to W$ fixing Ran$(\kappa)$.
 Thus $W$ is the reasonable complexification.  \end{proof}

 Basic aspects of the  real version of the theory of operator space tensor products are stated in Ruan's paper \cite{ROnr}.  We 
 will amplify on his very terse remarks, for example discussing  
  functoriality of the complexification for the three tensor products in Section 1.5  in  \cite[Section 1.5]{BLM}: the minimal, projective, and Haagerup operator space tensor product.
 The operator space projective tensor product  is designed to linearize real bilinear maps that are {\em jointly completely bounded} in  the sense of
 1.5.11 of \cite{BLM}.   One can see in several ways that a `jointly completely contractive' real  bilinear map 
$u : X \times Y \to Z$ of real operator spaces extends uniquely to a jointly completely contractive complex bilinear map 
$u : X_c \times Y_c \to Z_c$ (e.g.\ this follows from the next lemma).  As in the complex case $(E \overset{\frown}{\otimes} F)^* \cong
CB(E,Y^*)$ completely isometrically (see the last page of \cite{ROnr} and   (1.51) in \cite{BLM}). 
 
The Haagerup tensor product of real operator spaces is defined  e.g.\ as in 1.5.4 in \cite{BLM} so as  to have the universal property
of linearizing bilinear maps that are completely bounded in  the sense of
Christensen and Sinclair (called simply  completely bounded in e.g.\ 1.5.4 in \cite{BLM}).   As in the complex case, if $X_i \subset Y_i$ completely isometrically then 
$X_1 \otimes_{\rm h} X_2 \subset Y_1 \otimes_{\rm h} Y_2$ completely isometrically (this is the  `injectivity' of the 
Haagerup tensor product).  There are several known proof strategies for this, but it also follows from the 
complex case and the complete isometry beginning the third paragraph of the next proof.     The proof of the CSPS theorem in \cite[Theorem 1.5.7]{BLM} is 
just as in the complex case.   
  The other results in \cite[Section 1.5]{BLM}, e.g.\ the `selfduality' of  the Haagerup tensor product, are also 
 just as in the complex case.

\begin{lemma} \label{hts}  For real operator spaces $X$  and $Y$ we have $(X \otimes_{\rm min} Y)_c \cong X_c \otimes_{\rm min} Y_c$ and  $(X \overset{\frown}{\otimes} Y)_c \cong X_c \overset{\frown}{\otimes} Y_c$ completely isometrically. 
  Similarly  $(X  \otimes_{\rm h} Y)_c \cong X_c  \otimes_{\rm h} Y_c$ completely isometrically. 
  In particular, a completely contractive (in the sense of
Christensen and Sinclair)  real  bilinear map 
$u : X \times Y \to Z$ of real operator spaces extends uniquely to a completely contractive (in same sense) complex bilinear map 
$u : X_c \times Y_c \to Z_c$.  
\end{lemma}
 
   \begin{proof}     If $\otimes_{\rm min}$ is the uncompleted minimal operator space tensor product we have $$(X \otimes_{\rm min} Y)_c \subset CB(Y^*,X)_c = CB((Y^*)_c,X_c) = CB((Y_c)^* , X_c).$$ 
Thus the algebraic identification of $(X \otimes_{\rm min} Y)_c$ and $X_c \otimes_{\rm min} Y_c$  is a complete isometry.   
A similar principle shows that $(X \overset{\frown}{\otimes} Y)_c \cong X_c \overset{\frown}{\otimes} Y_c$ completely isometrically.  Indeed 
$$(X_c \overset{\frown}{\otimes} Y_c)^* \cong CB(X_c, (Y_c)^*) \cong CB(X_c, (Y^*)_c) \cong CB(X,Y^*)_c 
\cong ((X \overset{\frown}{\otimes} Y)^*)_c$$ completely isometrically, so that 
the algebraic identification of (the uncompleted) operator space projective tensor products  $X_c \overset{\frown}{\otimes} Y_c$ and 
$(X \overset{\frown}{\otimes} Y)_c$ is a complete isometry.   

Let $u  : X \times Y \to Z$ be a completely contractive (in the sense of
Christensen and Sinclair)  real  bilinear map 
and let $u^c : X_c \times Y_c \to Z_c$ be the unique complex bilinear extension.  Let $\theta : Z_c \to M_2(Z)$ be the map taking $x + i y$ to the matrix
in (\ref{ofr}).
For $n \in \Ndb$ and $a \in M_n(X_c), b \in M_n(Y_c)$ the reader can check that 
$\theta_n(u_n^c(a, b)) = u_{2n}(\theta_n(a), \theta_n(b))$.   Hence 
$$\| u_n^c(a, b) \| = \| \theta_n(u_n^c(a, b)) \| = \| u_{2n}(\theta_n(a), \theta_n(b)) \| \leq \| \theta_n(a) \| \, \| \theta_n(b)) \| = \| a \| \, \| b \| .$$
Thus $u^c$ is  completely contractive.  

It is easy to deduce from the latter and the universal property
that the canonical map $X  \otimes_{\rm h} Y \to X_c  \otimes_{\rm h} Y_c$ is a complete isometry. 
To show that $(X  \otimes_{\rm h} Y)_c \cong X_c  \otimes_{\rm h} Y_c$ completely isometrically it suffices by the `injectivity' of the 
Haagerup tensor product, and of $X \mapsto X_c$, to assume that $X$ and $Y$ are finite dimensional.   Clearly $\otimes : X \times Y \to 
X  \otimes_{\rm h} Y$ is completely contractive, so by facts above we obtain a completely contractive map
$X_c \times Y_c \to (X  \otimes_{\rm h} Y)_c$, and we see that the canonical map
$X_c  \otimes_{\rm h} Y_c \to (X  \otimes_{\rm h} Y)_c$  is completely contractive.   
Hence the canonical map
$X^*_c  \otimes_{\rm h} Y^*_c \to (X^* \otimes_{\rm h} Y^*)_c$  is completely contractive.    Dualizing this we get a 
complete contraction  $$(X^* \otimes_{\rm h} Y^*)_c^* \to (X^*_c  \otimes_{\rm h} Y^*_c)^* \cong X_c  \otimes_{\rm h} Y_c,$$ the latter
by the `selfduality' above the lemma.  Also by this `selfduality' we have  $X \otimes_{\rm h} Y \cong (X^* \otimes_{\rm h} Y^*)^*$. Complexifying this
last  isomorphism, and composing with the map in the last displayed equation,
one may see that the  canonical map $(X \otimes_{\rm h} Y)_c \to X_c  \otimes_{\rm h} Y_c$ is a complete contraction.
Thus  $(X  \otimes_{\rm h} Y)_c \cong X_c  \otimes_{\rm h} Y_c$ completely isometrically. 
   \end{proof}  
   
Nearly all of the real version of the results in \cite[Section 1.6]{BLM} are valid by the same arguments.  
We end this Section discussing the real versions of a couple of  results there that are not so clear.  

If $X, Y$ are weak* closed real subspaces of $B(H)$ and $B(K)$ we define the normal spatial tensor product $X \bar{\otimes} Y$ to be the weak* closure in $B(H \otimes K)$ of the span of 
the operators $x \otimes y$ for $x \in X, y \in Y$.  Define the  {\em normal Fubini tensor product} 
$X \otimes_{\mathcal F} Y$ to be the
subspace of $B(H \otimes K)$ of elements whose `left slices' are in $X$ and `right slices' are in $Y$ (see bottom p.\ 134 in \cite{ER}).

\begin{lemma} \label{Fub}    If $X$ and $Y$ are weak* closed real subspaces of $B(H)$ and $B(K)$ then there is a canonical weak* continuous completely isometric isomorphism 
$$(X_* \overset{\frown}{\otimes} Y_*)^* \; \cong \; X \otimes_{\mathcal F} Y,$$ 
which carries the weak* closure of the canonical copy of $X \otimes Y$  in $(X_* \overset{\frown}{\otimes} Y_*)^*$
onto the normal spatial tensor product $X \bar{\otimes} Y$.  

In addition, $X_c \bar{\otimes} Y_c$ is a reasonable complexification of  $X \bar{\otimes} Y$;
indeed   $X \bar{\otimes} Y$ may be identified with the fixed points of the period 2 automorphism 
$\theta_X \otimes \theta_Y$ on $X_c \bar{\otimes} Y_c$. 
 \end{lemma} 

 \begin{proof}   The proof of the complex version of the first assertion, which is Theorem 7.2.3  in \cite{ER}, works in the real case.
  
 Consider the canonical map
 $$X \bar{\otimes} Y \; \overset{\kappa}{\rightarrow} \; X_c \bar{\otimes} Y_c.$$
 Here $\kappa(x \otimes y) = x \otimes y$ for $x \in X, y \in Y$.
 Note that  $\kappa$ is weak* continuous and  complete isometric since it is a restriction of the canonical map
$$B(H) \bar{\otimes} B(K)  \; = \; B(H \otimes K) \; \overset{\kappa}{\rightarrow} \; B(H_c) \bar{\otimes} B(K_c)   \; = \; B(H_c \otimes K_c) .$$
(One may use Lemma \ref{comw}  here and below too if desired.)
Note also that $\theta_X \otimes \theta_Y$ is the  weak* continuous
restriction of $\theta_{B(H)} \otimes \theta_{B(K)}$ to $X_c \bar{\otimes} Y_c$, and is thus easily seen to be a well defined period 2 automorphism 
on $X_c \bar{\otimes} Y_c$.  
Let $$Q(z) = \frac{1}{2}(z  + (\theta_X \otimes \theta_Y)(z)) , \qquad z \in X_c \bar{\otimes} Y_c .$$ This is a weak* continuous 
idempotent,
and Ran$(Q)$ is the set of fixed points of $\theta_X \otimes \theta_Y$. 
It is easy to check that $Q \circ \kappa = \kappa$ first on elementary tensors, and then on $X \bar{\otimes} Y$ by continuity and density.
On the other hand $Q((x+iy) \otimes (x' + i y'))$ is easily seen to be in Ran$(\kappa)$, so since the latter is weak* closed
we have  Ran$(Q) \subset {\rm Ran}(\kappa)$ again by continuity and density.  Thus Ran$(Q) =  {\rm Ran}(\kappa)$ is  the set of fixed points of $\theta_X \otimes \theta_Y$.
The result follows from this and \cite[Proposition 2.1]{BCK}.
 \end{proof}

The next result is quite important, and contains  the real case of a celebrated result of 
Effros and Ruan: 

\begin{theorem} \label{well}    If $M$ and $N$ are real von Neumann algebras then the operator space projective tensor product 
$M_* \overset{\frown}{\otimes} N_*$ is the operator space predual of  the von Neumann algebra $M \bar{\otimes} N$. 
In addition, 
 $M \bar{\otimes} N$ is real $*$-isomorphic to the fixed points of the period 2 automorphism 
$\theta_M \otimes \theta_N$ on $M_c \bar{\otimes} N_c$. 
 \end{theorem} 

 \begin{proof}   That $M_* \overset{\frown}{\otimes} N_*$ is the operator space predual of  the normal spatial tensor product  $M \bar{\otimes} N$
 follows exactly as in the complex case (Theorem 7.2.4  in \cite{ER}), but using 
 Lemma \ref{Fub} above and the real version of the commutation theorem for tensor products
 in  \cite[Theorem 4.4.3]{Li}.  
 The last  assertion follows from  the previous result.   \end{proof}  
 
 Finally we consider the formula $L^\infty(\Omega,\mu) \bar{\otimes} Y = L^\infty(\Omega,\mu; Y)$ in 1.6.6 in \cite{BLM}, 
 for a dual operator space $Y$ with separable 
 predual.  In this case again the proof  referenced there (from Sakai's book) works in the real case, since it relies on the Dunford-Pettis theorem and the well known
 fact that $L^1(\Omega,\mu) \hat{\otimes} Y = L^1(\Omega,\mu; Y)$ (due to Grothendieck), 
 both of which are  valid in the real case. 
 (Actually we only need the first ten lines of Sakai's proof,  the rest is covered by our earlier discussion.)  This concludes our discussion of Chapter 1 in \cite{BLM}.  
 
 \section{Operator algebras} \label{oa}

 Before we discuss technical issues arising in  \cite[Chapter 2]{BLM} we say a few things about real Banach algebras.
  The first point is that the Neumann lemma and its variants, and the norm formulae 
 often  accompanying the Neumann lemma, are valid verbatim in unital real operator algebras via complexification.   This is because of the uniqueness of the operator space complexification, and the fact that the canonical map
   $A_c \to A$ is a contraction.  The basic spectral theory of real Banach algebras may be found e.g.\ in \cite{Good, Li}.  
 Nonzero multiplicative  linear  functionals (characters) $\chi$ on a real operator algebra $A$ need to be complex valued in general 
 to get a sensible theory (see \cite{Good,Li}).  If $A$ is unital then $\chi$ is  contractive.   If $A$ is not unital then $\chi$ extends to  a unital 
 character on $A^1$.  So $\chi$ is  (completely)  contractive.   Thus the Gelfand transform is a contraction, and the characters on 
 commutative unital real operator algebras are in bijective correspondence with the maximal ideals \cite{Good,Li}.  There is a natural involutive homeomorphism 
 on the spectrum of a commutative real $C^*$-algebra (see Example (3) at the end of 
\cite[Section 5.1]{Li}).
 There is therefore a good functional calculus in real $C^*$-algebras that works in the expected way 
 provided that one uses continuous functions on the spectrum that are involutive with respect to the just mentioned  involution on the spectrum
 (the real $C^*$-algebra generated by a normal element $a$ corresponds via the 
 Gelfand transform to such involutive functions on the spectrum).  See e.g.\ \cite{MMPS} for a recent survey on some aspects of real Banach algebras. 
  
 It seems to be well known that  Cohen's factorization theorem works for real Banach modules over real approximately unital Banach algebras,
 with the same proof.  Namely, 
 for a nondegenerate real left Banach $A$-module $X$  over a real approximately unital Banach algebra, if $x \in X$ of norm $< 1$ there exist
 $a \in {\rm Ball}(A)$ and $y  \in X$ of norm $< 1$ with $x = ay$. 
 Since we do not have a reference on hand we also mention that the basic proof of Cohen's theorem uses the Neumann lemma, which as we said
 works in the real case with the same norm inequality consequences,
 and builds the  element  $a$ from convex combinations of the 
 cai.  Similarly the  element  $y$ is also constructed using real methods, 
 so that 
 we still have $a  \in {\rm Ball}(A)$ and $y \in X$ with desired norm inequality.
 
 The following result is somewhat well known in the complex case, and has various known proofs, for example using a result of Kaplansky on minimal algebra norms on $C_0(K)$,
 or see e.g.\ \cite[Theorem A.5.9]{BLM}. The real case seems to be new,  but may be proved similarly using Kaplansky's result applied to the algebra 
 generated by $x^*x$.   We give a longer but selfcontained and novel route.
  
 \begin{theorem} \label{cid}    Let $\theta : A \to B$ be a contractive homomorphism from a real $C^*$-algebra into a real Banach algebra.
Then $\pi(A)$ is norm closed, and it possesses an involution with respect to which it is a real $C^*$-algebra.
 Moreover, $\pi$ is then a $*$-homomorphism into this $C^*$-algebra. If $\pi$ is one-to-one then it is an isometry. \end{theorem} 

 \begin{proof}  Assume first that $A$ is commutative.  By replacing $B$ by $\overline{\pi(B)}$ we may assume if one wishes 
  that $B$ is commutative, so Arens regular, although
 this is not necessary.  
 By extending to the bidual we may assume that  $A = M$ is a commutative  real $W^*$-algebra and $\pi$ is weak* continuous.  (Note that if $\pi^{**}$ is a 
 $*$-homomorphism into a $C^*$-algebra then $\pi$ is also, and its range is closed and is a $C^*$-algebra.)     Quotienting by the kernel, a weak* closed ideal,
 we may assume that $\pi$ is one-to-one.    
 By \cite[Theorem 3.1]{BCK} we may assume  that 
 $M = L^\infty(X,\Cdb) \oplus^\infty  L^\infty(Y,\Rdb)$.   Let $p$ be the projection in $M$ corresponding to the first summand. 
  We claim that $\pi(pM) = \pi(L^\infty(X,\Cdb))$ and $\pi((1-p) M) = \pi(L^\infty(Y,\Rdb))$ are closed and are $C^*$-algebras, so that $\pi(pM)  \oplus^\infty \pi((1-p) M)$ is a 
  $C^*$-algebra.      
  
   If $e$ and $f$ are mutually orthogonal projections in $M$ then 
  $\pi(e)$ and $\pi(f)$ are contractive nonzero idempotents with zero product, and are norm 1.   Hence it is easy to see that $\pi$ is isometric on real linear combinations
  of mutually orthogonal projections in $M$.  Since such real linear combinations are norm dense in
  $L^\infty(Y,\Rdb)$, it follows that $\pi$ is isometric on $(1-p) M$.  So  $\pi((1-p) M)$ is closed and has involution making it a 
  commutative  real $W^*$-algebra with its original norm, and $\pi$ is  a $*$-homomorphism.   
    
  We now consider the complementary space  $p M = L^\infty(X,\Cdb)$. Let $D$ be the closure of $\pi(L^\infty(X,\Cdb))$.   By the argument above $\pi$ is isometric 
  on $(pM)_{sa}$.  A similar argument will  show it isometric  on all of $pM$.   If $J = \pi(i)$ then $J^2 = -I$, and since $J$ and $\pi(-i)$ are contractive with product 1, they have norm 1. Note that 
    $$\| (\cos \theta I + J \sin \theta) x \| = \| \pi(cos \theta \cdot 1  + i \cdot \sin \theta) x \| \leq \| \pi(cos \theta \cdot 1  + i 1 \cdot \sin \theta) \| \| x \| \leq \| x \| $$
  since $\pi$ is contractive and $|cos \theta  + i \, \sin \theta | = 1$.  So  $D$ is a complex Banach space with $i x = J x$ for $x \in D$
  (by the simple criterion in the second paragraph of Section 3).
  Indeed it clearly is a  complex Banach algebra (c.f.\ \cite[Proposition 6.2]{MMPS}). 
  Now $\pi(i f) = J \pi(f) = i \pi(f)$, so that $\pi$ is a complex linear contractive homomorphism.  By the argument in the last paragraph but taking complex combinations 
     we have that $\pi$ is isometric on $p M$,
  and $\pi(pM)$ is closed and possesses an involution with respect to which it is a complex, hence real, commutative $C^*$-algebra, and $\pi$ is  a $*$-homomorphism. 
     Let $q = \pi(p)$, the identity of the latter $C^*$-algebra.

  For  $x, y \in M$ we have $$\| q \pi(x) \| \leq \| q \pi(x) + (1-q) \pi(y) \| \leq \| p x + (1-p) y \| = \max \{ \| px \| , \| (1-p) y \| 
  \} .$$
  Since $\theta$ is isometric on the two pieces the latter equals $\max \{ \| q \pi(x) \| , \| (1-q) \pi(y) \| \}$.  Thus $\pi(M) = \pi(pM)  \oplus^\infty \pi((1-p) M)$
  isometrically, and is norm closed and possesses an involution with respect to which it is a real $C^*$-algebra.
  Also  $\pi$ is  an isometric  $*$-homomorphism.

  In the general case, again $\theta^{**} :
A^{**} \to B^{**}$ is a unital contraction, and it is easy to see that it is a homomorphism  with respect to the left Arens product on
$B^{**}$ say.   So we may assume that $A$ is a real von Neumann algebra.
Indeed if the range of $\theta^{**}$ is a $C^*$-algebra then $\theta$ is a $*$-homomorphism by \cite[Theorem 2.6]{BT}, so has closed range
which is a $C^*$-algebra.  
As above we may assume that $\theta$ is one-to-one.  Claim: $\theta$ is an isometry.
By the commutative case, $\theta$ is an isometry on selfadjoint elements since they generate a commutative unital $C^*$-algebra.
For $a \in A$ let $a = v|a|$ be the polar decomposition in $A$ (see \cite[Proposition 4.3.4 (2)]{Li}).  
 Then $$\| \theta(a) \| = \| \theta(v|a|) \| = \| \theta(|a|) \| = \| |a | \|,$$ 
since $\theta(v), \theta(v^*)$ are contractions and $|a| = v^* v|a|$.  So $\theta$ is an isometry on $A$.   Hence $\pi(A)$ is closed, and has an involution 
with respect to which it is a $C^*$-algebra. 
   \end{proof} 

We thank Roger Smith for a discussion on the general case in the last paragraph, and for some suggestions.  

We now survey the real case of the remaining results in Chapter 2 of \cite{BLM} (see also the last paragraphs of the Introduction).   As said elsewhere, the selfadjoint elements in a real
 operator algebra are not necessarily the hermitian elements,
although  a selfadjoint element $a$ is positive iff $\varphi(a) \geq 0$ for real states $\varphi$.    States are thoroughly discussed in
\cite{BT}.  One may add to that that the real states $\varphi$ on an approximately unital real operator algebra $A$ are precisely the real parts of complex states  on $A_c$ (such as $\varphi_c$), or of a complex $C^*$-algebra generated by $A_c$ (using \cite[Lemma 4.13]{BT} in the approximately unital case).  However the real parts of two different such complex states may coincide on $A$. A similar fact holds in the Jordan operator algebra case. 
The results 2.1.5--2.1.8 and the relevant Appendix A.6 in \cite{BLM} rely in places on Cohen factorization, but as we mentioned above the latter is valid
in the real case. 
Lemma 2.1.9 and 2.1.18 in the real case appear as Proposition 4.3 and Lemma 4.12 in \cite{BT}.
Indeed most results in Section 2.1  of \cite{BLM} are just as in the complex case, or are explicitly 
in \cite{Sharma,BT}.  E.g.\ the unitization $A^1$ was studied in \cite[Section 3]{Sharma}
where it was checked that $(A^1)_c \cong (A_c)^1$ completely isometrically.    For the $\cU(X)$ construction from Section 2.2  of \cite{BLM} it is clear that 
$\cU(X)_c \cong  \cU(X_c)$ completely isometrically and as operator algebras (just as we saw the analogous result for $\cS(X)_c$ in the introduction).   
Sharma abstractly characterized real operator algebras in \cite{Sharma}.  The rest of the results in Section 2.2--2.5 in   \cite{BLM}
are virtually unchanged in the real case, with the exception (as in Chapter 1) being the complex interpolation results in Section 2.3. 
We mention in particular the fact that  the quotient of a real operator algebra $A$ by a closed two sided ideal is a real operator algebra, which  follows as in 
 \cite[Proposition 2.3.4]{BLM} or from the complex case and the 
 earlier complete isometry $A/I \subset (A/I)_c \cong A_c/I_c$ \cite[Lemma 5.12]{Sharma}. 

The theory of the
 left, right and two-sided multiplier algebras of an approximately unital operator algebra  $A$ in \cite[Section 2.6]{BLM} is just as in the complex case.   The following follows 
 from Lemma \ref{mltsco} but we include the short direct proof.
 
  \begin{lemma} \label{mulr}  For a real approximately unital operator algebra $A$,    
  a reasonable complexification of $LM(A)$ is $(LM(A_c), T \mapsto T_c)$.  Similarly $M(A)_c = M(A_c)$and $RM(A)_c = RM(A_c)$
  completely isometrically isomorphically. \end{lemma}

  \begin{proof} 
  We have (completely isometrically as operator algebras) 
 $$LM(A_c) = \{ \eta \in (A_c)^{**} \cong (A^{**})_c : \eta A_c \subset A_c \},$$ which is
 $\{ \zeta + i \xi  \in (A^{**})_c : \zeta A \subset A, \xi A \subset A \}.$   Since the latter
 is reasonable it  is 
  $LM(A)_c.$  
  It follows that $(LM(A_c), T \mapsto T_c)$ is a reasonable complexification of $LM(A)$; also  $T \in LM(A)$ iff $T_c \in LM(A_c)$. 
  The others are similar.
   \end{proof} 

Also 2.6.14--17 (for real selfadjoint UCP maps),  and the real version of Section 2.7, hold in the real case.
The real version of the important 
 \cite[Theorem 2.7.9]{BLM} holds by complexification.  For the readers convenience
  we walk quickly through the details as being representative of such arguments:
 
 \begin{theorem} \label{doa}   Let $M$ be a real operator algebra which is a dual operator space. Then the product on $M$ is separately weak* continuous, and $M$ is a dual operator algebra. That is, there exists a  real Hilbert space $H$ and a  weak* continuous completely isometric homomorphism $\pi : M \to  B(H).$ \end{theorem} 

 \begin{proof}   Indeed if $A$ is a real operator  algebra with an operator space predual $Y$, then $A_c$ is
a complex operator  algebra with an operator space predual $Y_c$.  Also $A$ is a weak* closed real subalgebra of its complexification by  our comments above
on \cite[Section 1.4]{BLM}.   
 Thus by \cite[Theorem 2.7.9]{BLM} there exists a (complex hence) real Hilbert space $H$ and a weak* continuous 
completely isometric homomorphism $\pi : A \to B(H)$ with $\pi(A)$ a weak* closed real  operator algebra on $H$.  \end{proof}

We call these the {\em real dual operator algebras}.  We remark that some of the real theory of dual operator algebras also
follows from the involutive approach as in \cite[Section 4]{BWinv}.

We end this section with a few words about the completely isomorphic version of the theory. 
As mentioned in the introduction we only check some selected results in Chapter 5 of 
\cite{BLM}.   The results in Section 5.1 are valid in the real case with the usual exception of the complex interpolation result (5.1.10). 
As we said in Section 2 above, the main result in Section 5.2
follows from the complex case by complexification.  Similarly, the matching result for modules 5.2.17 is valid in the real case with
unchanged proof.

 \section{Operator modules} \label{om}
 The real analogues of nearly all results in Chapter 3 in \cite{BLM} are unproblematic, although some of these results need some facts about operator space multipliers 
 from Section \ref{ospm} above in place of their complex variants.  The complexification of what are called operator modules, $h$-modules, Hilbert modules,
 and matrix normed modules in 
 that chapter are complex operator modules, $h$-modules, and matrix normed modules (using for example facts mentioned in Section \ref{os} about the complexification of tensor products
 in reference to Section 1.5 in \cite{BLM}). 
 For example using this principle the proof in  Section \ref{absc} that the real version of the Christensen-Effros-Sinclair theorem characterizing 
operator modules in \cite[Theorem 3.3.1]{BLM}) may be rephrased in terms of complexifying all spaces and using  the last assertion of 
Lemma \ref{hts}.  

 Items 3.1.11 and 3.5.4-3.5.5 in \cite{BLM} uses Cohen factorization, but as we mentioned above the latter
 and several accompanying results in Appendix A.6 in \cite{BLM} are  valid in the real case.
 Theorem 3.2.14 there uses Appendix A.1.5.   However the latter result follows in the real case by complexification. 
 Indeed if $W$  is as  in Appendix A.1.5 and $T$ is as in the fifth line of A.1.5 
 then we have $W_c$ weak* closed in $B(H_c)$, and clearly $(W_c)^{(\infty)} = (W^{(\infty)})_c$.  We also clearly  have $T_c \zeta \in [(W^{(\infty)})_c \, \zeta]$ 
 for $\zeta \in (H_c)^{(\infty)} = (H^{(\infty)})_c$.  
 Since $(W_c)^{(\infty)}$ is reflexive by A.1.5 in the complex case, 
 we have $T_c  \in (W_c)^{(\infty)} 
  = (W^{(\infty)})_c$, so that $T \in W^{(\infty)}$. 
 
 Much of the real version of Section 3.7 is classical, in the literature in some form, and is largely due to E. Behrends and his collaborators 
 \cite{Behr,HWW}.  Some of this is explicitly discussed in Section \ref{ospm}, mostly in Theorem \ref{centmin} and its proof. 
  Some other results in Section 3.7  follow immediately from the complex case by complexification. 
 The real versions of results in Section 3.8 in \cite{BLM} follow with the same proofs or by complexification. 
 
  \section{Operator spaces and injectivity, etc.}  \label{inj4} 

 We turn to the few remaining parts of Chapter 4  in \cite{BLM}:  The motivational Section 4.1 is largely a review of the classical Shilov and Choquet boundary of complex function spaces.
 There is a  literature of these boundaries  for real function spaces but we will say no more here on this topic since our goal is the more
 general real operator space case. 
  Early results in Section  4.2 and on the real injective envelope and $C^*$-envelope
 are covered in \cite{Sharma} (some are mentioned without proof
 in Ruan's papers on real operator spaces). 
 Indeed items up to 4.2.7 and 4.2.11 are in \cite{Sharma} or are obvious in the real case.
Corollary 4.2.8 there is done  in Theorem 4.2 
of \cite{BCK},  and Corollary  4.2.9 follows by complexification.  
  The first part of 4.2.10 is as in the complex case. 
   That $I(M_{m,n}(X)) \cong M_{m,n}(I(X))$ is generalized in 4.6.12 in \cite{BLM} 
  whose proof is unchanged in the real case.  It is proved in  \cite{BCK} that
  a  complex operator space is real injective if and only if it is complex injective.
    
  The real C*-envelope is discussed in \cite{Sharma} and \cite[Section 4]{BCK}.
  The real versions of the important facts in 4.3.2 and 4.3.6 in \cite{BLM} about the $C^*$-envelope 
are true, but some of these use the real versions of facts from Chapters 1 and 2 of   \cite{BLM}  that are 
discussed above.  
       Examples (1) and (2) in 4.3.7 are essentially unchanged except that one must use the classification of
  finite dimensional real $C^*$-algebras in \cite[Theorem 5.7.1]{Li}.   We have not checked 4.3.8--4.3.11 although we would guess that
  4.3.8 and the Dirichlet algebra results are unchanged.   These results seem like a possibly interesting project.
   Theorem 4.4.3  in   \cite{BLM} is valid, and its Corollary 4.4.4 was already in \cite{ROnr}.  The remaining 
  results in Section 4.4 in \cite{BLM}  are unchanged, or follow by complexification 
  as in the case of Youngson's theorem 4.4.9 and the Corollary after that.

   In view of the importance to real operator spaces and their complexifications of the space $\bar{X}$ discussed in and around Lemma  \ref{lem2},
   the following is often quite useful.

   \begin{proposition}  \label{ciop} If $X$ (resp.\ $A$) is a complex  operator space (resp.\ algebra), then 
    $$\overline{(A^1)}= (\bar{A})^1 \; , \; I(\bar{X}) = \overline{I(X)}  \; , \;  {\mathcal T}(\bar{X}) =  \overline{{\mathcal T}(X)}  \; , \;   C^*_e(\bar{A}) =  \overline{C^*_e(A)}.$$
\end{proposition}

\begin{proof}    If $(I(X),j)$ is a (complex operator space)  injective envelope of a complex operator space $X$, then 
$(\overline{I(X)},\bar{j})$ is an injective envelope of $\bar{X}$.   Indeed as indicated in  \cite[Proposition 2.1]{BWinv}, a routine diagram chase (i.e.\ applying the functor above
the proposition to the `universal injectivity and rigidity diagrams') shows that
 $\overline{I(X)}$ is 1)\ injective, and 2)\ has the `rigidity property', so is an injective envelope of $\bar{X}$.  
 
 In complicated situations it is sometimes very useful to view the above 
  in terms of a containing $C^*$-algebra $B$ where it is much easier to see what the `bar' is doing in terms of the $C^*$-algebra adjoint and opposite
  (indeed recall from the proof of Lemma \ref{lem2} that $\bar{B} = B^\circ$, the `opposite' $C^*$-algebra). 
  We take the time to spell this out in a bit more detail. 
 Namely consider the $C^*$-algebra $B = I({\mathcal S}(X))$ which has $I(X)$ as its 1-2 corner $Z$ (see e.g.\ 4.4.2 in \cite{BLM}).   
Viewed in this way 
 it is clear that $I(X)$ is a ternary system or TRO (that is,  $Z Z^* Z \subset Z$).  In terms of 
 the `opposite' $X^\circ$ 
 and `adjoint' $X^\star$ constructions from 1.2.25 in \cite{BLM} and  
 \cite[Proposition 2.1]{BWinv} we have $\bar{X} = (X^\circ)^\star$. 
 We view $X^\circ$ and $I(X)^\circ$ as the appropriate subspaces of the `opposite $C^*$-algebra' $B^{\circ}$.
 Then $\bar{X}$ and $\overline{I(X)}$ are the `adjoints' of the latter subspaces of $B^{\circ}$ (by adjoint we mean the $C^*$-algebra
 involution of the $C^*$-algebra $B^{\circ}$).    ]
 Similarly, for any complex subTRO $Z$  of $B$ the canonical map $Z \to \bar{Z}$ is therefore a conjugate linear TRO morphism and `complete isometry'.
 Viewed in this way it is easy to see that  $\overline{I(X)}$ is again a TRO, and indeed  $I(\bar{X}) = \overline{I(X)}$.  Similarly, 
  $(\overline{{\mathcal T}(X)},\bar{j})$ is a ternary  envelope of $\bar{X}$.
  That is, $\overline{{\mathcal T}(X)}$  is the subTRO of $\overline{I(X)}$ generated by 
   $\bar{j}(\bar{X})$.  
  This can also be deduced 
  from \cite[Proposition 2.1]{BWinv}:
  $${\mathcal T}(\bar{X}) = {\mathcal T}((X^\circ)^\star) 
  = {\mathcal T}(X^\circ)^\star = ({\mathcal T}(X)^\circ)^\star =\overline{{\mathcal T}(X)}.$$
  Similarly, if $X = A$ is a complex  operator algebra then one may deduce in the same way from \cite[Proposition 2.1]{BWinv}
  that $\overline{(A^1)}= (\bar{A})^1$, and 
  $(\overline{C^*_e(A)}, \bar{j})$ is a $C^*$-envelope of $\bar{A}$. 
\end{proof} 

It follows for example 
from the last result and Lemma \ref{lem2} that for a complex operator space $X$ we have $I(X_c)  \cong I(X \oplus^\infty \bar{X})  \cong I(X) \oplus^\infty \overline{I(X)} \cong I(X )_c$. 

Concerning  the Banach-Stone type theorem 4.5.13, in the case that $A$ is unital  there are many proofs of this 
  in the literature which still work in the real case. 
  E.g.\ the method in 8.3.13, whose proof is unchanged in the real case.  Or the unital case follows from 
  the real unital $C^*$-algebra  Banach-Stone type theorem \cite[Theorem 4.4]{RComp} (which generalizes 
  the simple  \cite[Corollary 1.3.10]{BLM}), by passing to the 
  injective or $C^*$-envelope.   We have several more general complex Banach-Stone theorems in e.g.\ 
  \cite{BNp,BNjp}, and it would be interesting to see which of these are true in the real case (see also 
 the paragraph after Proposition 2.10 in \cite{BT}).   
 We remark that Ph.\ D.\ student Dylan Phelps is pursuing some of these directions under our co-direction with Labuschagne.

\section{$C^*$-modules and TRO's} \label{last} 

Real Hilbert $C^*$-modules over a  real $C^*$-algebra are defined almost exactly as in the complex case--
see Definition 2.4 in \cite{HW}.
 Let $A$  be a real $C^*$-algebra and $V$  a real right $C^*$-module over $A$. Then it is known (see e.g.\ \cite{HW}) that  there is an 
$A_c$-valued inner product on $V_c = V + iV$ 
$$\langle v+iw , x+i y \rangle = \langle v , x \rangle + \langle w ,  y \rangle + i(\langle v , y \rangle - \langle w , x \rangle ),$$ 
extending the original $A$-valued inner product on $V$  
such that $V_c$ is a complex $C^*$-module over $A_c$.   This complexification  
is `reasonable' with respect to the canonical norm: $\| v - i w \| = \| v +  i w \|$. 
Similarly $M_n(V)$ is a real right  
$C^*$-module over  $M_n(A)$, and its complexification is reasonable. 
With these matrix norms $V$ is a real operator space, a real subspace of $V_c$ with the  
latter regarded as a real operator space. 
Thus the  
$C^*$-module $V_c$ with its canonical operator space structure
is a (completely) reasonable complexification. 
It follows from Ruan's theorem that the 
this canonical operator space structure  coincides with the unique operator space complexification of $V$.    

If $\{ Y_t \}$ is a family of real right $C^*$-modules over $A$ then the $C^*$-module sum
 $\oplus_t \, Y_t$ is a real right $C^*$-module over $A$ 
as in the complex case and $$(\oplus_t \, Y_t)_c \cong \oplus_t \, (Y_t)_c.$$   Indeed one may define 
$\oplus_t \, Y_t$ to be the closure of the real span of the copies of $Y_t$ in $\oplus_t \, (Y_t)_c.$  Indeed the inherited inner product on this
closed real subspace lies in $A$. 
Viewing the latter as an operator space as in \cite[Section 8.2]{BLM},  it is clear that 
$\oplus_t \, (Y_t)_c$ is a reasonable operator space complexification of $\oplus_t \, Y_t$.   Indeed since $A_c$ is a reasonable complexification of $A$ it is easy to see that 
$$\| (y_t) + i (z_t) \|^2 = \| \sum_t \, \langle y_t + i z_t  , y_t + i z_t \rangle \| = \| \sum_t \, \langle y_t - i z_t  , y_t -i  z_t \rangle \| = \| (y_t) - i (z_t) \|^2$$
for $y_t, z_t \in Y$.   We leave it to the reader to check the matricial case of this computation, that is that $\oplus_t \, (Y_t)_c$ is a completely 
reasonable complexification.   By Ruan's theorem we deduce that $(\oplus_t \, Y_t)_c \cong \oplus_t \, (Y_t)_c.$

If $T : Y \to Z$ is adjointable then it is easy to check that $T_c$ is adjointable with 
$(T_c)^* = (T^*)_c$.   The set $\bB_A(Y)$ of such adjointable $T : Y \to Y$ is therefore 
$*$-isomorphic to a real  $*$-subalgebra of the complex $C^*$-algebra $\Bdb_{A_c}(Y_c)$. 
It is useful to know that

\begin{lemma} \label{isaj} If $Y, Z$ are  real right $C^*$-modules over $A$ and $T : Y \to Z$ is an $A$-module map.  Then 
$T_c : Y_c \to Z_c$  is adjointable if and only if  $T$ is adjointable. \end{lemma} \begin{proof}  The one direction is obvious.
If $T_c : Y_c \to Z_c$  is adjointable then
 $$\langle y, (T_c)^*(z) \rangle = \langle T_c y, z \rangle = \langle T y, z \rangle \in A, \qquad y \in Y, z \in Z.$$ 
 Writing $(T_c)^*(z) = y_1 + i y_2$ we easily see by examining the part of the left side of the last equation that is in $iA$ that
 $\langle y,   y_2 \rangle = 0$.  Setting $y = y_2$ gives $\langle y_2 , y_2 \rangle = 0$.  Hence $(T_c)^*(Y) \subset Z$, so that 
 $T$ is adjointable on $Y$. \end{proof}

Also, $\bB_A(Y)$ is closed in $B_B(Y)$ as in the complex case.   So $D = \bB_A(Y)$ is a 
real unital $C^*$-algebra, and we may view it as a  real  $C^*$-subalgebra of $E = \Bdb_{A_c}(Y_c)$.   
 For $y, x,w \in Y$ we have $$(| y \rangle \langle y |) x + i (| y \rangle \langle y |) w
= y \, \langle y , x \rangle + i y \, \langle y , w \rangle .$$
Also within   $\Bdb_{A_c}(Y_c)$ we have 
$| y \rangle \langle y | (x + iw) = y (\langle y , x \rangle + i \langle y , w \rangle)$.
Thus $(| y \rangle \langle y |)_c = | y \rangle \langle y |$ in $\Bdb_{A_c}(Y_c)$ which is a positive operator.
Now $D_c$ may be viewed as a complex $C^*$-subalgebra of $E_c$, and 
$d \in D$ is in $D_+$ iff $d \in (D_c)_+$ iff $d \in (E_c)_+$ iff $d \in E_+$, where the latter are the positive operators in $E$,with $E$ considered as a 
real $C^*$-algebra.   However by \cite[Proposition 5.2.2]{Li} this is equivalent to $d$ being positive in $E$ in the usual sense.
Thus $| y \rangle \langle y | \geq 0$ in $\bB_A(Y)$.   It follows that $Y$ is a real left  $C^*$-module over $\bB_A(Y)$, and over 
$\bK_A(Y)$, where the latter is the closure of the span of the $| y \rangle \langle z |$ for $y, z \in Y$.

Since every real $C^*$-algebra has an increasing contractive approximate identity (c.f.\ the proof of \cite[Proposition 5.2.4]{Li}), 
8.1.3 and 8.1.4(1) in \cite{BLM} hold. 
Since Cohen's factorization theorem works for real Banach modules over real $C^*$-algebras as we said at the start of our discussion of 
Chapter 2 in \cite{BLM},  
8.1.4 (2) holds and the rest of 8.1.4.  Lemma 8.1.5 may be proved by complexification. 
Proposition 8.1.6 holds with the same proof, but the argument after that proof fails. This is the 
assertion that the inner product on a $C^*$-module $Y$ is completely determined by, and may be recovered from, the Banach module structure.
To see this, simply complexify and use the complex case of this remark.

The definitions and arguments in 8.1.7 holds except for the point involving the polarization identity. 
However this point is proved in \cite{HW}: A surjective isometric $A$-module map between 
real right $C^*$-modules is an (adjointable) unitary $A$-module map.  Indeed we have the real version of
8.1.8; and 
8.1.9, 8.1.11, 8.1.14 and 8.1.15 hold with the same  arguments (in some of these using again that Cohen's factorization theorem holds in the real case).  
 We do not have the results in 8.1.12 and 8.1.13 of course.   
 It is easy to see that the real variant of \cite[Proposition 8.1.16 (1)--(3)]{BLM} holds.     Item (4) there fails (take $Y = B$ a real unital $C^*$-algebra
 whose selfadjoint elements do not span $B$).  
The theory of the linking $C^*$-algebra in  8.1.17--8.1.19 is unchanged.   We have:

 \begin{theorem}  \label{csmhom}  For real right $C^*$-modules $Y, Z$ over $A$  we have   
$$\Kdb_{A_c}(Y_c,Z_c) \cong  \Kdb_{A}(Y,Z)_c \, , \; \; \; \; \Bdb_{A_c}(Y_c,Z_c)  \cong \Bdb_{A}(Y,Z)_c  \, , \; \; \; 
 (B_A(Y,Z))_c  \cong B_{A_c}(Y_c,Z_c)$$
completely isometrically.  If $Y = Z$ then the first two of these isomorphisms are $*$-isomorphisms,
and the last is also as algebras.  For the linking algebra, $\cL(Y)_c \cong \cL(Y_c)$ $*$-isomorphically.  \end{theorem} 

\begin{proof}   By replacing the spaces by the $C^*$-module sum $Y \oplus Z$ we may assume that $Y = Z$. 
 Then since $Y_c$ is an  equivalence $\Kdb_{A_c}(Y_c)$-$A_c$-bimodule and 
 also an  equivalence  $\Kdb_{A}(Y)_c$-$A_c$-bimodule, it  follows that 
 $\Kdb_{A_c}(Y_c) = \Kdb_{A}(Y)_c$.  The action of $\Kdb_{A}(Y)_c$ on $Y_c$ is as
 $$(R + iS)(x+iy) = Rx -Sy + i(Ry +Sx) = R_c(x+iy)  + i S_c(x+iy) = (R_c + i S_c)(x+iy) .$$   
 If $R_c = i S_c$ then $Rx + iRy = -Sy + i Sx$ for all $x, y \in Y$ and $R = S = 0$.    Thus
 $\Kdb_{A_c}(Y_c)$ may be identified as $C^*$-algebras with $\{ R_c + i S_c \in \bB_{A_c}(Y_c) : R, S \in \Kdb_{A}(Y) \}$.   
 We have 
 $$(B_A(Y))_c \cong (LM(\Kdb_{A}(Y)))_c \cong LM(\Kdb_{A}(Y))_c) \cong LM(\bK_{A_c}(Y_c)) \cong B_{A_c}(Y_c),$$
 and similarly $(\Bdb_A(Y))_c \cong \Bdb_{A_c}(Y_c).$. Finally, 
 $$\cL(Y)_c \cong \bK_A(Y \oplus A)_c \cong \bK_{A_c}((Y \oplus A)_c ) \cong \bK_{A_c}(Y_c \oplus A_c) \cong \cL(Y_c) ,$$ 
 and $\cL(Y)$ is clearly a $*$-subalgebra of $\cL(Y_c)$. 
\end{proof}
  
  Taking $Y = C_n(A)$ and considering the natural $*$-homomorphism $M_n(A) \to \Kdb_A(C_n(A))$ we see that the norm of a matrix 
  in $M_n(A)$ is given by the formulae in \cite[Corollary 8.1.13]{BLM}.  However the equivalence with (ii) in that result is not valid even if $A = \bR$
  and $a$ is selfadjoint and $n = 2$.
   
 Corollary 8.1.20 is true but needs a change at a couple of points of the proof.  One first replaces the appeal to Proposition 8.1.16 (4)
 to an appeal to a fact about the diagonal mentioned after Theorem \ref{cid}.
 This shows that the contractive unital homomorphism $\pi : Z(M(B)) \to CB_B(Y)$, which
 may be viewed as mapping into $LM(\Kdb_B(Y))$, actually maps into
 $\Delta(LM(\Kdb_B(Y)) = M(\Kdb_B(Y)) = \Bdb_B(Y)$.    
 
 One may avoid the  problematic parts of the proof of Corollary 8.1.21 in the real case by  instead proving the 
 result by complexification.   Indeed in that proof $P_c$ will be adjointable on
 $Y_c$ by the complex version of this result, so that $P$ is adjointable on
 $Y$ by Lemma \ref{isaj}. 
 In the proof of  8.1.22 we may apply the argument for contractivity of $R$ to $R_c = \sum_i \, (Q_i)_c$, to see that $R_c$ and hence 
 $R$ is contractive.  The rest of that proof works in the real case. 
  Items  8.1.23--25 are true with the same  proof, and 
 8.1.26 follows by complexification.  The proof of Theorem 8.1.27 (1) uses the polar decomposition of maps on complex $C^*$-modules
 from the source cited there.  Inspecting that source shows that the argument there works in the real case.  
 
 Item 8.2.1 is valid in the real case, and is important (some of this we have stated already).  
 If $T : Y \to Y$ is a contractive $A$-module map then 
by Theorem \ref{csmhom}  we have that $T_c$ is contractive, hence completely contractive by the complex case.
 Thus \cite[Proposition 8.2.2]{BLM} holds.   We also deduce that $\| T \| = \| T_c \| = \| T_c   \|_{\rm cb},$ by a property of the operator space complexification. 
 Also the paragraph after that Proposition, and the items in 8.2.3 and 8.2.4, hold in the real case.  
  Items 8.2.5--8.2.7 are valid in the real case, although one needs to check that  the facts cited in 8.2.5 from other sources 
  hold in the real case.   For this it is useful to note that if $a \in A$ is a  strictly positive element in a real $C^*$-algebra $A$ 
  (so that $a = a^* \geq 0$ and $\varphi(a) > 0$ for all real states $\varphi$ on $A$), then 
  $a$ is  strictly positive  in $A_c$. 
  This is because if $\psi$ is a state on $A_c$ then $\psi(a) \geq 0$ and Re$\psi_{A}$ is a state on $A$ so that $\psi(a) > 0$.   If $a$ is a strictly positive element then $a^{\frac{1}{n}}$ will be a countable cai, as usual. 
  
 Essentially all of the rest of Section 8.2 is true in the real case with the same arguments.     For example, item 8.2.24 is true in the real case, and we note that lattices (1)--(4) there 
  are also in correspondence to the analogous lattices in the complexified spaces.    Similar statements hold for the real variant of 8.2.25.    Some of the real versions of basic facts about  TRO's in \cite[Section 8.3]{BLM} are discussed at the end of the introduction to \cite{BCK}.   The rest of the items in Section 8.3 
  are also valid with the same proofs.    
  
   In Section 8.4 a word needs to be said about  formula (8.17)  since the proof of this there uses 
  the span of Hermitian elements.  However (8.17) may be seen directly: let $C$ be the algebra on the right side of (8.16).  Then 
  $\Delta(C)$ is a $*$-subalgebra of $\Delta(CB_{\cF}(\cT(X)) = \Bdb_{\cF}(\cT(X))$.  Thus $\Delta(C)$ is a $*$-subalgebra of the 
  algebra on the right side of (8.17).   Conversely the latter algebra is clearly contained in $\Delta(C)$. 
  
 Similarly, the proof that (iii) implies (i) in  Theorem 8.4.4 needs to be altered in the real case: if $T$ and $R$ are as in (iii), 
  note that $T_c$ and $R_c$ satisfy the analogous relation for the complexifications.   Thus $T_c \in \cA_{\ell}(X_c) \cong \cA_{\ell}(X)_c$ and hence $T \in \cA_{\ell}(X)$.
  (Alternatively, $T_c$ is left multiplication by an element in $I_{11}(X_c)$, so can be identified with left multiplication by $a + ib$ for $a, b \in \cM_{\ell}(X) \subset I_{11}(X)$.
  Since  $(a^* + i b^*) j_c(X_c) \subset  j_c(X_c)$ we see that $a \in \cA_{\ell}(X)$.)
    Items 8.5.1--8.5.21 are as in the complex case,     although in 8.5.20 the $M$-ideal case seems extremely difficult and only recently proved in a paper referred to in the Acknowledgements below.  In the complete $M$-ideal case the     results follow by 
    working in the complexification (we are not sure if all  $M$-ideals are complete $M$-ideals here, except under extra conditions).   
  
  Reaching  fatigue, we will say nothing about the real case of the remaining results 8.5.22--40,  except to say that most of these items seemed unproblematic
  on our first pass through them, but 
  we have not checked them carefully. 
   This would be a nice project since many of these results are extremely important. 
  Most of these items are not operator space results, 
  but rather are von Neumann algebraic, so one might expect that a few of these results will be in the literature. 
    
  Finally, the real version of most results in the first part of Section 8.6 are essentially already 
   noted by Ruan at the end of his two papers on real operator spaces, although he leaves the details to the reader.
  Indeed  he points out these operator space properties hold on 
  $X$ if and only if they hold on $X_c$, so that the real case of results in 8.6.1--8.6.3 follow from their complex version.
  
\subsection*{Acknowledgements}  Supported by a Simons Foundation Collaboration Grant/Travel Support for Mathematicians.   We thank M. Neal and R. R. Smith for discussions. Much of the present paper concerns real versions of the theory of operator spaces and algebras represented in \cite{BLM} or from around the time of that text.   Recently we have extended some more current aspects of this theory, such as e.g.\ real positivity, and some aspects of complex Jordan operator algebras, to the real setting.
Also, the questions we raised on $M$-ideals in real operator algebras and real TRO's were solved after the acceptance of this paper.   See  two recent papers by D. P. Blecher, M. Neal, A. M. Peralta and S. Su, (e.g.\ ``$M$-ideals, yet again: the case of real JB$^*$-triples'', preprint 2024).

\end{document}